\documentclass[a4paper,10pt]{amsart}

\usepackage{amsthm}                   
\usepackage{amsmath}                  
\usepackage{amsfonts}                 
\usepackage{amssymb}                  
\usepackage{mathrsfs}                 
\usepackage{xy}
\usepackage{graphicx,subfigure}
\xyoption{all}

\usepackage{epstopdf}

\theoremstyle{definition}
\newtheorem{definition}{Definition}[section]
\newtheorem{remark}[definition]{Remark}

\newtheorem{characterization}[definition]{Characterization}
\theoremstyle{plain}
\newtheorem{theorem}[definition]{Theorem}
\newtheorem{lemma}[definition]{Lemma}

\newtheorem{proposition}[definition]{Proposition}
\newtheorem*{theorem*}{Theorem}

\numberwithin{equation}{section}

\newcommand{\eps}{\varepsilon}

\begin{document}
\author{Beno\^{i}t Loridant}\author{Shu-qin Zhang}
\address{Chair of Mathematics and Statistics, University of Leoben, Franz-Josef-Strasse 18, A-8700 Leoben, Austria}
\email{benoit.loridant@unileoben.ac.at}
\email{zhangsq\_ccnu@sina.com}
\title[Topology of a class of $p2$-crystallographic tiles]{Topology of a class of $p2$-crystallographic replication tiles}

\date{\today}
\thanks{Supported by project I1136 and by the doctoral program W1230 granted by the Austrian Science Fund (FWF)}

\maketitle

\begin{abstract}
We study the topological properties of a class of planar crystallographic replication tiles. Let $M\in\mathbb{Z}^{2\times2}$ be an expanding matrix with characteristic polynomial $x^2+Ax+B$ ($A,B\in\mathbb{Z}$, $B\geq 2$) and ${\bf v}\in\mathbb{Z}^2$ such that $({\bf v},M{\bf v})$ are linearly independent. Then the equation
$$MT+\frac{B-1}{2}{\bf v} =T\cup(T+{\bf v})\cup (T+2{\bf v})\cup
\cdots\cup(T+(B-2){\bf v})\cup(-T)
$$
defines a unique nonempty compact set $T$ satisfying $\overline{T^o}=T$. Moreover, $T$ tiles the plane by the crystallographic group $p2$ generated by the $\pi$-rotation and the translations by integer vectors. It was proved by Leung and Lau in the context of self-affine lattice tiles with collinear digit set that $T\cup (-T)$ is homeomorphic to a closed disk if and only if $2|A|<B+3$. However, this characterization does not hold anymore for $T$ itself. In this paper, we completely characterize the tiles $T$ of this class that are homeomorphic to a closed disk.
\end{abstract}

\section{Introduction}

A \emph{crystallographic replication tile}  with respect to a
crystallographic group $\Gamma\subset\textrm{Isom}(\mathbb{R}^n)$ is a nonempty compact set $T\subset{\mathbb R}^n$ that is the closure of its interior ($\overline{T^o}=T$) and satisfies the following properties.
\begin{itemize}
\item[$(i)$] There is an expanding affine mapping $g:{\mathbb
R}^n\rightarrow{\mathbb R}^n$ such that $g\circ\Gamma\circ
g^{-1}\subset\Gamma$, and a finite collection $\mathcal{D}\subset\Gamma$ called \emph{digit set} such that
$$
g(T)=\bigcup_{\delta\in\mathcal{D}}\delta(T).
$$
\item[$(ii)$] The family $\{\gamma(T); \gamma\in\Gamma\}$ is a \emph{tiling} of ${\mathbb R}^n$. In other words, $\mathbb{R}^n=\bigcup_{\gamma\in\Gamma}\gamma(T)$ and $\gamma(T^o)\cap\gamma'(T^o)=\emptyset$ for distinct elements $\gamma,\gamma'\in\Gamma$. 
\end{itemize}

There is a vast literature dealing with the lattice case, \emph{i.e.},  when $\Gamma$ is isomorphic to $\mathbb{Z}^n$: criteria exist to check basic properties, such as the tiling property~\cite{LagariasWang96a}, connectedness~\cite{KiratLau00} or, in the planar case ($n=2$),  homeomorphy to a closed disk (\emph{disk-likeness}). For instance, Bandt and Wang recognize disk-like self-affine lattice tiles by the number and location of the neighbors in the tiling~\cite{BandtWang01}, and Lau and Leung characterize all the disk-like tiles among the class of self-affine lattice tiles with collinear digit set~\cite{LauLeung07}. A powerful tool in the study of topological properties is the \emph{neighbor graph}: it gives a precise description of the boundary of the tile in terms of a \emph{graph directed iterated function system} (\emph{GIFS}). Akiyama and the first author elaborated a boundary parametrization method by making extensive use of the neighbor graph~\cite{AkiyamaLoridant11}. Algorithms allow to determine the neighbor graph for any given tile $T$~\cite{ScheicherThuswaldner02}, while it is usually difficult to deal with infinite classes of tiles. However, Akiyama and Thuswaldner computed the neighbor graph for an infinite class of planar self-affine lattice tiles associated with canonical number systems and used it to characterize  the disk-like tiles among this class~\cite{AkiyamaThuswaldner05}. Methods relying on the neighbor graph were extended to crystallographic replication tiles in~\cite{LoridantLuo09,LoridantLuoThuswaldner07}. 

If $T$ is a crystallographic replication tile, the associated digit set $\mathcal{D}$ must be a complete set  of right coset representatives of the subgroup $g\circ\Gamma\circ g^{-1}$. On the other side, if $T$ is a nonempty compact set $T\subset{\mathbb R}^n$ satisfying $(i)$ and $\mathcal{D}$ is a complete set  of right coset representatives of the subgroup $g\circ\Gamma\circ g^{-1}$, Gelbrich proves that there is a subset $\Gamma'\subset\Gamma$ called \emph{tiling set} such that the family $\{\gamma(T); \gamma\in\Gamma'\}$ is a tiling of ${\mathbb R}^n$.  Under these conditions, it is not known in general whether the tiling set $\Gamma'$ is a \emph{subgroup} of the crystallographic group $\Gamma$, contrary to the lattice case (see~\cite{LagariasWang96b}). However, the first author defined  in~\cite{Loridant12} the \emph{crystallographic number systems}, in analogy to the canonical number systems from the lattice case (see \emph{e.g.}~\cite{Katai95}). This gives  a way to produce classes of crystallographic replication tiles whose tiling set is the whole group $\Gamma$. An infinite class of examples given in~\cite{Loridant12} reads as follows. Let $p2$  be the planar crystallographic group generated by the translations $a(x,y)=(x+1,y)$, $b(x,y)=(x,y+1)$ and the $\pi$-rotation $c(x,y)=(-x,-y)$. Moreover, for $A,B\in\mathbb{Z}$ satisfying $|A|\leq B\geq2$, let $g$ be the expanding mapping  defined on $\mathbb{R}^2$ by
$$g(x,y)=\left(\begin{matrix}
0 &-B\\
1&-A\\
\end{matrix}
\right)
\left(\begin{matrix}
x\\
y\\
\end{matrix}
\right)
+
\left(\begin{matrix}
\frac{B-1}{2}\\
0\\
\end{matrix}
\right).
$$
Then the equation
$$g(T)=T\cup\left(T+\left(\begin{array}{c}1\\0\end{array}\right)\right)\cup
\cdots\cup\left(T+\left(\begin{array}{c}B-2\\0\end{array}\right)\right)\cup(-T)
$$
defines a crystallographic replication tile whose tiling set is the whole group $p2$. This tiling property follows from the crystallographic number system property only for $A\geq -1$, as stated in~\cite{Loridant12}, but we will deduce it for all values of $A$. Moreover, we will obtain topological information on $T$ by comparing it with the self-affine lattice tile $T^l$ defined by
$$\left(\begin{matrix}
0 &-B\\
1&-A\\
\end{matrix}
\right) T^{\ell}= T^{\ell}\cup\left(T^{\ell}+\left(\begin{array}{c}1\\0\end{array}\right)\right)\cup
\cdots\cup\left(T^{\ell}+\left(\begin{array}{c}B-1\\0\end{array}\right)\right).
$$
In fact, for fixed $A
\text{ and }B$, the tile $T^l$ is a translation of $T\cup(-T)$, as shown in~\cite{Loridant12}. It follows from Leung and Lau's result~\cite{LauLeung07}  on self-affine tiles with collinear digit set that $T^l$ is disk-like if and only if $2|A|-B< 3$. However, it was noticed in~\cite{Loridant12} that it can happen that $T^l$ is disk-like while $T$ is not disk-like (see the examples of Figure~  \ref{A2B3_subfigures} and Figure~ \ref{A3B4_subfigures}). The current paper will establish exactly for which parameters $A,B$ this phenomenon occurs. For $2|A|-B<3$, the associated lattice tile $T^l$ is disk-like and a result of Akiyama and Thuswaldner~\cite{AkiyamaThuswaldner05} on canonical number system tiles will allow us to estimate the set of neighbors of $T$. Finding out the disk-like tiles for parameters satisfying $2|A|-B<3$ will then rely on the construction of the associated neighbor graphs for the whole class. For $2|A|-B\geq3$, a purely topological argument will enable us to prove that the associated tiles are not disk-like.

 Our results easily generalize to a broader class of crystallographic replication tiles, closely related to the class of self-affine tiles with collinear digit set as studied by Leung and Lau in~\cite{LauLeung07}. Therefore, we are able to show the following classification theorem. 

\begin{theorem*}
Let $A,B\in\mathbb{Z}$ satisfying $|A|\leq B\text{ and }B\geq2$, 
$M\in\mathbb{Z}^{2\times2}$ a matrix with characteristic polynomial $x^2+Ax+B$ and let ${\bf v}\in\mathbb{Z}^2$ such that $({\bf v},M{\bf v})$ are linearly independent. Let $T$ be the crystallographic replication tile defined by 
$$MT+\frac{B-1}{2}{\bf v}=T\cup(T+{\bf v})\cup (T+2{\bf v})\cup
\cdots\cup(T+(B-2){\bf v})\cup(-T).
$$
Then the following statements hold.
\begin{itemize}
\item Suppose that $2|A|-B\geq 3$. Then $T$ is not disk-like.
\item Suppose that $2|A|-B<3$. Then one of the following cases occurs:
\begin{enumerate}
\item If $A\in\{-2,~-1,~0,~1\}, B\geq 2$ or $A=2, ~B=2$, then $T$ is disk-like.
\item If $A\geq 2, ~B\geq3$ or $A\leq -3, ~B\geq 4$, then $T$ is not disk-like. 

\end{enumerate}
\end{itemize}
\end{theorem*}

The paper is organized as follows. In
Section~\ref{sec_basic}, we give basic definitions on crystallographic groups and general properties of the class of crystallographic replication tiles under consideration.  Sections~\ref{sec_neigh1} and~\ref{sec_neigh2} are devoted to the construction of the neighbor graphs for part of this class. They will be the main tool for our topological study. In Section~\ref{sec_disklikepos} and Section \ref{sec_disklikeneg}, we characterize the disk-like tiles among our class for the range of parameters $A,B$ satisfying $2|A|-B<3$. In Section~\ref{sec_nondisk}, we show that $T$ is not disk-like for all parameters satisfying $2|A|-B\geq 3$. Finally, Section~\ref{sec_example} illustrates the theorem by examples.


\section{Preliminaries}\label{sec_basic}

\subsection{Basic definitions}
Let us recall some definitions and facts about tilings and crystallographic replication tiles (\emph{crystiles} for short).

A tiling of $\mathbb{R}^2$ is a cover of the space by nonoverlapping sets, \emph{i.e.}, such that the interiors of two distinct sets of the cover are disjoint. Some particular tilings use a single tile $T$ with $\overline{T^{\circ}}=T$
and a family $\Gamma$ of isometries of $\mathbb{R}^2$ such that 
$$\mathbb{R}^2=\bigcup_{\gamma\in\Gamma} \gamma(T).$$
Assume that $\Gamma$ contains $id$, the identity map of $\mathbb{R}^2$. Then $T=id(T)$ is called the \emph{central tile} of the tiling. Also, two distinct tiles are said to be \emph{neighbors} if they have common points. The neighbor set of $T$ is then given by 
$$\mathcal{S}=\{\gamma\in\Gamma\setminus \{id\}; \gamma(T)\cap T\neq \emptyset\}.$$
It is symmetric and it generates $\Gamma$ ($\Gamma=\langle\mathcal{S}\rangle$). The tiles considered in this paper will be compact and the tilings locally finite, \emph{i.e.}, every compact set intersects finitely many tiles of the tiling. Thus $\mathcal{S}$ is always a finite set here.

Among the possible neighbors of a tile, we consider the following two kinds of neighbors.  Two neighbors are  called \emph{vertex neighbors} if they have only one common point. Two neighbors are \emph{adjacent neighbors} if the interior of their union contains a point of their intersection. 
The \emph{adjacent neighbor set} $\mathcal{A}\subset\mathcal{S}$ is then defined as the set of adjacent neighbors of the identity: 
$$\mathcal{A}=\{\gamma\in\mathcal{S}; T\cap \gamma(T) \cap(T\cup\gamma(T))^\circ\neq \emptyset\}.$$
The neighbor (resp. adjacent neighbor) set of a tile $\gamma(T)~(\gamma\in\Gamma)$ is equal to $\gamma\mathcal{S}$ (resp. $\gamma\mathcal{A}$). 

We will deal with families $\Gamma$ of isometries  that are \emph{crystallographic groups} in dimension 2, \emph{i.e.}, discrete cocompact subgroups $\Gamma$ of the group Isom($\mathbb{R}^2$) of all isometries on $\mathbb{R}^2$ with respect to some metric. By a theorem of Bieberbach (see \cite{Burckhardt47}), a crystallographic group $\Gamma$ in dimension $2$ contains a group $\Lambda$ of translations isomorphic to the lattice $\mathbb{Z}^2$, and the quotient group $\Gamma/\Lambda$, called \emph{point group}, is finite. There are $17$ nonisomorphic such groups. However, in this paper, we will mainly consider 
the following crystallographic $p2$-groups.

\begin{definition}\label{P2}
Let $a(x,y)=(x+1,y),~ b(x,y)=(x,y+1), ~c(x,y)=(-x,-y)$. Then a $\emph{p2-group}$ is a group of isometries of  $\mathbb{R}^2$ isomorphic to the subgroup of Isom($\mathbb{R}^2$) generated by the translations $a,b$ and the $\pi$-rotation $c$.
\end{definition}

For example, the standard $p2$-group $\Gamma$ has the form 
\begin{equation}\label{Gamma}
\Gamma=\{a^pb^qc^r; p,q\in\mathbb{Z}, r\in\{0,1\}\},
\end{equation} 
and it is a crystallographic group.
We will call a tiling with respect to a $p2$-group a $\emph{p2}$-\emph{tiling}, and a tiling with respect to a lattice group (\emph{i.e.}, for which the point group only contains the class of the identity map of $\mathbb{R}^2$)  a \emph{lattice tiling}.

We will be concerned with self-replicating tiles constructed in the following way. We refer the reader to~\cite{Gelbrich94, LoridantLuoThuswaldner07} for further information about these tiles.
\begin{definition}
A \emph{crystallographic replication tile} with respect to a crystallographic group $\Gamma$ is a compact nonempty set $T\subset \mathbb{R}^n$ with the following properties:
\begin{itemize}
\item The family $\{\gamma(T); \gamma\in\Gamma\}$ is a tiling of $\mathbb{R}^n$.
\item There is an expanding affine map $g: \mathbb{R}^n\longrightarrow\mathbb{R}^n$ such that $g\circ\Gamma\circ g^{-1}\subset\Gamma$ and there exists a finite collection $\mathcal{D}\subset \Gamma$ called \emph{digit set} such that  
$$g(T)=\bigcup_{\delta\in\mathcal{D}}\delta(T).$$
\end{itemize}
\end{definition}

\subsection{Lattice tiling and $p2$-tiling}
From now on, $\Gamma$ is the standard $p2$-group defined \eqref{Gamma}. We recall that an expanding affine map $g$ in $\mathbb{R}^n$ has the form $g(x)=Mx+t$, where $t\in\mathbb{R}^n$ and $M$ is an $n\times n$-matrix whose eigenvalues all have modulus greater than 1. 

We consider a special class of $p2$-crystallographic replication tiles, closely related to the class of self-affine tiles with collinear digit set studied by Leung and Lau in~\cite{LauLeung07}. For $A,B\in\mathbb{Z}$, $B\geq 2$, let $\tilde{M}\in\mathbb{Z}^{2\times2}$ be a matrix with characteristic polynomial $x^2+Ax+B$. Then $\tilde{M}$ is expanding, \emph{i.e.}, its eigenvalues are greater than $1$ in modulus, if and only if 
$|A|\leq B$. Moreover, let ${\bf v}\in\mathbb{Z}^2$ such that $({\bf v},\tilde{M}{\bf v})$ are linearly independent and $\tilde{a}(x,y)=(x,y)+{\bf v}$. We set $\tilde{g}({\bf x})=\tilde{M}{\bf x} +\frac{B-1}{2}{\bf v}$. Then one can check that the digit set
{\small $$\tilde{\mathcal{D}}=\{id,\tilde{a},\ldots,\tilde{a}^{B-2},c\}
$$}
is a complete set  of right coset representatives of the subgroup $\tilde{g}\circ\Gamma\circ \tilde{g}^{-1}$. Therefore, by a result of Gelbrich~\cite{Gelbrich94}, the equation 
{\small $$\tilde{g}(\tilde{T}) =\bigcup_{\delta\in\tilde{\mathcal{D}}}\delta(\tilde{T})
$$}
defines a unique nonempty compact set $\tilde{T}(A,B)=\tilde{T}$ satisfying $\overline{\tilde{T}^o}=\tilde{T}$. The purpose of this paper is the topological study of the tiles $\tilde{T}$. In fact, we can reduce this study to the following subclass. Let
{\small \begin{equation}\label{expand-map}
g(x,y)=\left(\begin{matrix}
0 &-B\\
1&-A\\
\end{matrix}
\right)
\left(\begin{matrix}
x\\
y\\
\end{matrix}
\right)
+
\left(\begin{matrix}
\frac{B-1}{2}\\
0\\
\end{matrix}
\right)
\end{equation} }
and
{\small \begin{equation}\label{digit}
\begin{matrix}
\mathcal{D}=\{id, a,\dots,a^{B-2},c\} &  \text{ if } B\geq 3,\\
\mathcal{D}=\{id,c\}~~\quad\quad\quad\quad\quad& \text{ if } B=2,
\end{matrix}
\end{equation}}
which is a complete right residue system of $\Gamma/g\Gamma g^{-1}$. We denote by $T(A,B)=T$ the associated tile satisfying {\small $$g(T)=\bigcup_{\delta\in\mathcal{D}}\delta(T).$$}
\begin{lemma}In the above notations, let $C$ denote the matrix of change of base from the canonical base to the base $({\bf v},\tilde{M}{\bf v})$. Then 
$$\tilde{T}=CT.
$$
\end{lemma}
\begin{proof}
We have by definition
{\small $$\tilde{M}\tilde{T}+\frac{B-1}{2}{\bf v} =\tilde{T}\cup(\tilde{T}+{\bf v})\cup 
\cdots\cup(\tilde{T}+(B-2){\bf v})\cup(-\tilde{T}).
$$}
Using the equality $\tilde{M}=CMC^{-1}$ and multiplying the above equation by $C^{-1}$, we obtain
{\Small $$MC^{-1}\tilde{T}+\left(\begin{array}{c}\frac{B-1}{2}\\0\end{array}\right) =C^{-1}\tilde{T}\cup\left(C^{-1}\tilde{T}+\left(\begin{array}{c}1\\0\end{array}\right)\right)\cup
\cdots\cup\left(C^{-1}\tilde{T}+\left(\begin{array}{c}B-2\\0\end{array}\right)\right)\cup(-C^{-1}\tilde{T}).
$$ }
By uniqueness of the IFS-attractor, this leads to $C^{-1}\tilde{T}=T$.
\end{proof}
By this lemma, the topology of $\tilde{T}$ is the same as the topology of $T$, that is why the proofs will be written for the class of tiles defined by~(\ref{expand-map}) and~(\ref{digit}). The relation to self-affine tiles with collinear digit set now reads as follows. 
Let {\small $M=\left(\begin{matrix}
0 & -B\\
1 & -A\\
\end{matrix}
\right)\in\mathbb{Z}^{2\times2}$ and $$\mathcal{N}=\left\{ 
\left(\begin{matrix}
0\\
0\\
\end{matrix}
\right),
\left(\begin{matrix}
1\\
0\\
\end{matrix}
\right)
,\dots,
\left(\begin{matrix}
B-1\\
0\\
\end{matrix}\right)\right\}.$$ }

We denote by $T^{\ell}(A,B)=T^{\ell}$ the associated lattice tile satisfying {\small $$MT^{\ell}=\bigcup_{d\in\mathcal{N}}(T^{\ell}+d).$$}
Note that the crystallographic data $(p2,g,\mathcal{D})$ is very similar to the lattice data $(\mathbb{Z}^2, M, \mathcal{N})$. However, in the lattice case, we often prefer to consider the above translation vectors rather than the translation mappings $id,a,\dots,a^{B-1}$. Moreover,  \cite{Loridant12} also gave the following correspondence between the crystallographic tiles and the associated lattice tiles of the above class.
\begin{lemma}\label{twotiles}
With the above data, let $T$ satisfy
$g(T)=\bigcup_{\delta\in\mathcal{D}}\delta(T)$
and  $T^{\ell}$ satisfy $MT^{\ell}=\bigcup_{d\in \mathcal{N}}(T^{\ell}+d).$
Then 
{\small \begin{equation}
T^{\ell}=T\cup(-T)+(M-I_2)^{-1}
\left(\begin{matrix}
\frac{B-1}{2}\\
0
\end{matrix}\right),
\end{equation}}
where $I_2$ is the $2\times 2$ identity matrix.
\end{lemma}
Hereafter, we denote the crystallographic tile and lattice tile associated with the above data $(p2,g, \mathcal{D})$ and $(\mathbb{Z}^2,M,\mathcal{N})$ by $T$ and $T^{\ell}$, respectively.
\begin{lemma}\label{lem:tiling}$T$ is a crystallographic replication tile.
\end{lemma}
\begin{proof}
We only need to prove that the family $\{\gamma(T);\gamma\in\Gamma\}$ is a tiling of $\mathbb{R}^2$. We recall that $T$ has nonempty interior by a result of Gelbrich \cite{Gelbrich94}, because $\mathcal{D}$ is a complete set of right coset representatives of $g\circ\Gamma\circ g^{-1}$. Now, for $A\geq-1$, the family $\{T^{\ell}+z;z\in\mathbb{Z}^2\}$ is a tiling of $\mathbb{R}^2$, since the tile $T^{\ell}$ is associated to a quadratic canonical number system (see \emph{e.g.}~\cite{AkiyamaThuswaldner05}). This also holds for the tiles $T^{\ell}$ with $A\leq0$, as it is mentioned in~\cite{AkiyamaLoridant10} that changing $A$ to $-A$, for a fixed $B$, results in an isometric transformation for the associated tiles $T^{\ell}$ (see Equation~(\ref{relat_pos_neg})). Therefore, by Lemma~\ref{twotiles}, we just need to show that $T$ and $c(T)=-T$ do not overlap. This follows from the fact that $T$ has nonempty interior and satisfies the set equation
{\small $$T=g^{-1}(T)\cup g^{-1}(-T)\cup g^{-1}\circ a(T)\cdots\cup g^{-1}\circ a^{B-2}(T).
$$}
Indeed, each of the $B$ sets on the right side of this equation has two-dimensional Lebesgue measure $\alpha/B$, where $\alpha>0$ is the  two-dimensional Lebesgue measure of $T$. The total measure of the right side being equal to $\alpha$,  the sets can not overlap. 
\end{proof}
Note that for $-1\leq A\leq B$, the above lemma is also a consequence of the crystallographic number system property~\cite{Loridant12}.

\begin{remark} In the above proof,  we mentioned the easy relation~(\ref{relat_pos_neg})  between the lattice tiles $T^{\ell}$ associated to $A$ and $-A$. It turns out that no such easy relation can be found for the corresponding tiles $T$, and the topology may become different  when changing $A$ to $-A$  (see Section~\ref{sec_disklikeneg}, Figure \ref{fig:AMA}).   
\end{remark}
For the lattice data $(\mathbb{Z}^2,M,\mathcal{N})$, the following proposition is proved by Leung and Lau~\cite{LauLeung07}.

\begin{proposition}\label{disklattice}
Let $A \text{ and } B$ satisfy $|A|\leq B \text{ and } B\geq 2$. Then $T^{\ell}$ is homeomorphic to a closed disk if and only if  $2|A|<B+3$.
\end{proposition}

\section{The neighbor set of $T$ for $A\geq -1$ and $2A<B+3$}\label{sec_neigh1}
For the sake of simplicity, in Sections~\ref{sec_neigh1},~\ref{sec_neigh2} and~\ref{sec_disklikepos}   we will now restrict to the case  $A\geq -1$ and $2A<B+3$ and indicate in Section~\ref{sec_disklikeneg} the method to get the results for $A\leq -2$.  

In this section, we will derive an ``approximation" of the neighbor set $\mathcal{S}$ for $A\geq-1$, $2A<B+3$ from the relationship between the neighbor set of $T$ and the neighbor set of $T^{\ell}$. Akiyama and Thuswaldner prove the following characterization of the  neighbors of $T^{\ell}$ in~\cite{AkiyamaThuswaldner05}.

\begin{proposition}\label{neigh_lattice}
Let $\mathcal{S}^{\ell}$ denote the neighbor set of $T^{\ell}$. If $2A<B+3$ and $A\neq0$, then $\sharp\mathcal{S}^{\ell}=6$. In particular, if $A>0$, then $$\mathcal{S}^{\ell}=\{a^Ab,~a^{A-1}b,~a,~a^{-1},~a^{-A}b^{-1},~a^{-A+1}b^{-1}\};$$
if $A=-1$, we have 
$$\mathcal{S}^{\ell}=\{a^{-1}b, ~b,~a,~a^{-1},~ab^{-1},~b^{-1}\};$$
if $A=0$, we have 
$$\mathcal{S}^{\ell}=\{a,~a^{-1}, ab,a^{-1}b, ~ab^{-1}, a^{-1}b^{-1},~b,~b^{-1}\}.$$
\end{proposition}

The following lemma gives a first coarse estimate of the neighbor set of $T$ in terms of the neighbor set of $T^{\ell}$.
\begin{lemma}\label{neigh_crystile}
Let $\mathcal{S}, \mathcal{S}^{\ell}$ be the neighbor sets of $T$ and $T^{\ell}$, respectively, then $\mathcal{S}$ is a subset of $\mathcal{S}^{\ell}\cup\{c\}\cup{\mathcal{S}^{\ell}c}$, where $\mathcal{S}^{\ell}c=\{s\circ c;~~ s\in\mathcal{S}^{\ell}\}.$
\end{lemma}
\begin{proof}
Using Lemma \ref{twotiles}, we know that the  lattice tile is a translation of the union $T\cup c(T)$. Then it is easy to see that all possible elements of the neighbor set of $T$ are included in the union of the neighbor set of $T^{\ell}$, the $\pi$-rotation of the neighbor set of $T^{\ell}$ and the $\pi$-rotation itself. 
\end{proof}

From the above lemma, we know an upper bound for the number of neighbors of the $p2$-tile $T$.  The following proposition from~\cite{LoridantLuo09} gives a lower bound for this number.

\begin{proposition}\label{num_p2}
In a lattice tiling or a $p2$-tiling of the plane, each tile has at least six neighbors.
\end{proposition}

Let us now give the definition of the neighbor graph.
\begin{definition}(\cite{LoridantLuoThuswaldner07})\label{graph}
 Let $\Gamma$ be a group of isometries on $\mathbb{R}^2$, let $g$ be an expanding map and let $\mathcal{D}$ be a digit set associated with a given tile. For $\Omega\subset \Gamma$ we define the graph $G(\Omega)$ as follows. The states of $G(\Omega)$ are the elements of $\Omega$, and there is an edge 
$$\gamma\xrightarrow{\delta|\delta'}\gamma'\quad\text{  iff  }\delta^{-1}g\gamma g^{-1}\delta'=\gamma'  \text{ with }\gamma,\gamma' \in \Omega \text{ and } \delta,\delta' \in \mathcal{D}.$$
\end{definition}

The \emph{neighbor graph} $G(\mathcal{S})$ is very important in the present paper.

Recall that the neighbor set of $T$ is defined by $\mathcal{S}=\{\gamma\in\Gamma\setminus\{id\};~~T\cap\gamma(T)\neq\emptyset\}.$ Set $B_{\gamma}=T\cap\gamma(T)$ for $\gamma \in \Gamma$. The nonoverlapping property yields for the boundary of $T$ that $\partial T=\bigcup_{\gamma\in\mathcal{S}}B_{\gamma}$. 
Moreover using the above notation, the sets $B_{\gamma}$ satisfy the set equation (\cite{LoridantLuoThuswaldner07})
$$B_{\gamma}=\bigcup_{
\begin{matrix}
\delta\in \mathcal{D},\gamma'\in\mathcal{S}, \\
\exists~~\delta'\in\mathcal{D},\gamma\xrightarrow{\delta|\delta'}\gamma'\in G(\mathcal{S})\\
\end{matrix}}g^{-1}\delta(B_{\gamma'}).$$

 The following characterization is from \cite{LoridantLuoThuswaldner07}.

\begin{characterization}\label{B_r}
Let $t$ be a point in $\mathbb{R}^n$, $(\delta_j)_{j\in\mathbb{N}}\in\mathcal{D}^{\mathbb{N}}$ and $\gamma\in\mathcal{S}$. Then the following assertions are equivalent.
\begin{itemize}
\item $x=\lim_{n\rightarrow\infty}g^{-1}\delta_1\dots g^{-1}\delta_n(t)\in B_{\gamma}$.
\item There is an infinite walk in $G(\mathcal{S})$ of the shape
$$\gamma\xrightarrow{\delta_1|\delta_1'}\gamma_1\xrightarrow{\delta_2|\delta_2'}\gamma_2\xrightarrow{\delta_3|\delta_3'}\dots$$ for some $\gamma_i\in\mathcal{S}$ and $\delta'_i\in\mathcal{D}$.
\end{itemize}  
\end{characterization}

This means that for each $\gamma\in \mathcal{S}$, there is at least one infinite walk in $G(\mathcal{S})$ starting from the state $\gamma$.
We use this information to refine the estimate of the neighbor set of $T$ (compare with Lemma~\ref{neigh_crystile}).

\begin{lemma}\label{estimate_neighbor}
Let $\mathcal{S}$ be the neighbor set of the tile $T$ with respect to $(p2,g,\mathcal{D})$. Let $\mathcal{S}'=\mathcal{S}^{\ell}\cup\{c\}\cup{\mathcal{S}^{\ell}c}$. Then the following statements hold.
\begin{itemize} 
\item[(1)] For $A>0$, $\mathcal{S}\subset\mathcal{S}'\setminus\{ a^Ab,a^{-A}b^{-1},a^{-A}b^{-1}c \}$;
\item[(2)] For $A=-1$, $\mathcal{S}\subset\mathcal{S}'\setminus\{ a^{-1}b, ~b, ~ab^{-1}, ~b^{-1}, ~ab^{-1}c,~b^{-1}c \}$;
\item[(3)] For $A=0$, $\mathcal{S}\subset\mathcal{S}'\setminus\{ ab, a^{-1}b^{-1}, ~ab^{-1}, a^{-1}b, ~b,  ~b^{-1}, ~a^{-1}b^{-1}c,~ab^{-1}c,~b^{-1}c\}.$
\end{itemize}
In particular, $\mathcal{S}$ has at least $6$ but not more than $10$ elements.
\end{lemma} 
\begin{proof}
We know that $G(\mathcal{S})$ is a subgraph of  $G(\mathcal{S}')$ by Lemma \ref{neigh_crystile}.
The definition of the edges requires to calculate $g\mathcal{S}'g^{-1}=\{g\gamma g^{-1};~~\gamma\in\mathcal{S}'\}$ at first.
Let $p\text{ and }q$ be arbitrary elements in $\mathbb{Z}$. Recall that $g$ has the form  \eqref{expand-map}. Then
\begin{equation}\label{ab}
g a^p b^q g^{-1} \left(
\begin{matrix}
x\\
y\\
\end{matrix}
\right)=
\left(
\begin{matrix}
x\\
y\\
\end{matrix}
\right)+
\left(
\begin{matrix}
-qB\\
p-qA\\
\end{matrix}
\right),
\end{equation}

\begin{equation}\label{abc}
g a^p b^q cg^{-1} \left(
\begin{matrix}
x\\
y\\
\end{matrix}
\right)=
-\left(
\begin{matrix}
x\\
y\\
\end{matrix}
\right)+
\left(
\begin{matrix}
(1-q)B-1\\
p-qA\\
\end{matrix}
\right).
\end{equation}
Thus the following relations hold:
$$ga^Abg^{-1}=a^{-B},\quad ga^{-A}b^{-1}g^{-1}=a^{B},\quad ga^{-A}b^{-1}cg^{-1}=a^{2B-1}c.$$
We claim that there are no edges starting from the states $a^Ab,~~a^{-A}b^{-1},~~\text{and }a^{-A}b^{-1}c.$

Indeed, for $\delta,\delta'\in \mathcal{D}$, 
$$\delta^{-1}ga^Abg^{-1}\delta'=\delta^{-1}a^{-B}\delta'= \left\lbrace
\begin{matrix}
a^{-B}~~,&\delta=\delta'=id;~~\\
a^B~~~~~~,&\delta=\delta'=c;\quad\\
a^{-B}c~~,&\delta=id,~\delta'=c;\\
a^{B}c~~,&\delta=c,~\delta'=id;\\
a^{B-p+q},&\quad\quad\delta=a^p,~\delta'=a^q, ~1\leq p,q\leq B-2.\\
\end{matrix}
\right.$$
Therefore, $\delta^{-1}ga^Abg^{-1}\delta'$ is not an element of $\mathcal{S}'$, which means that there is no edge starting from $a^Ab$. The computation is  similar for $a^{-A}b^{-1},~~a^{-A}b^{-1}c$.
Hence, we obtain that $a^Ab,~~a^{-A}b^{-1},~~a^{-A}b^{-1}c$ are not elements of $\mathcal{S}$ by Characterization \ref{B_r}, which proves Item $(1)$.

For $A=-1$, by~\eqref{ab} and \eqref{abc} we know that

$\begin{matrix}
ga^{-1}bg^{-1}=a^{-B},\quad gbg^{-1}=a^{-B}b,\quad gab^{-1}g^{-1}=a^{B},\quad\quad\quad\\
\quad\quad\quad gb^{-1}g^{-1}=a^Bb^{-1},\quad gab^{-1}cg^{-1}=a^{2B-1}c,\quad gb^{-1}cg^{-1}=a^{2B-1}b^{-1}c.\\
\end{matrix} $\\
Similar computations as above show that   there is no edge starting from the states removed from $\mathcal{S}'$ in Item (2).

For $A=0$, we can also show that there is no edge starting from the states removed from $\mathcal{S}'$ in Item (3), since

$\begin{matrix}
ga^{-1}bg^{-1}=a^{-B}b^{-1},\quad gbg^{-1}=a^{-B},\quad gab^{-1}g^{-1}=a^{B}b,\quad\quad\quad\\
\quad\quad\quad gb^{-1}g^{-1}=a^B,\quad gab^{-1}cg^{-1}=a^{2B-1}bc,\quad gb^{-1}cg^{-1}=a^{2B-1}c,\\
\quad\quad\quad\quad gabg^{-1}=a^{-B}b,\quad ga^{-1}b^{-1}g^{-1}=a^Bb^{-1},\quad ga^{-1}b^{-1}cg^{-1}=a^{2B-1}b^{-1}c.\\
\end{matrix}$

Finally, by Proposition~\ref{num_p2} and the above discussion, we obtain   that the neighbor set of the crystile has at least $6$ but not more than $10$ elements because $\sharp\mathcal{S}'=13$ by Lemma \ref{neigh_crystile}.
\end{proof}

\section{The neighbor graph of $T$ for $A\geq -1$ and $2A<B+3$}\label{sec_neigh2}
In this section, we explicitly construct the neighbor graph. Throughout the whole section, we restrict to the case $A\geq -1$ and $2A<B+3$.
In Lemma \ref{estimate_neighbor}, we denoted by $\mathcal{S}'$ the set $\mathcal{S}'=\mathcal{S}^{\ell}\cup\{c\}\cup\mathcal{S}^{\ell}c$.  Now for $A>0$, let
$\mathcal{S}''=\mathcal{S}'\setminus\{ a^Ab,a^{-A}b^{-1},a^{-A}b^{-1}c \}$,
that is,
\begin{equation}\label{Pseudo-neigh}
\mathcal{S}''=\{a^{A-1}b,a,a^{-1},b^{-1},a^{1-A}b^{-1},c,a^Abc,a^{A-1}bc,ac,a^{-1}c,a^{1-A}b^{-1}c\}.
\end{equation}
For $A=0$, we set
\begin{equation}\label{Pseudo-neigh0}\mathcal{S}''=\{a, a^{-1}, c, ac, a^{-1}c, a^{-1}bc, bc,abc\},
\end{equation}
and for $A=-1$, 
\begin{equation}\label{Pseudo-neigh-1}\mathcal{S}''=\{a, a^{-1}, c, ac, a^{-1}c, a^{-1}bc, bc\}.
\end{equation}
 By Lemma \ref{neigh_crystile}, we know that $\mathcal{S}\subset\mathcal{S}''$.  
We call the graph $G(\mathcal{S}'')$ the \emph{pseudo-neighbor graph}. Tables $1, 2 \text{ and } 3$ show all information on $G(\mathcal{S}'')$. The last column indicates the parameters $A, B$ for which these edges exist. Furthermore, the pseudo-neighbor graphs for the cases $A\geq 3, B\geq 5$ are depicted in Figure \ref{Uni-graph}. The edges named by $(1),\dots, (13)$ are listed in Tables $1, 2 \text{ and } 3$. 

\begin{figure}[h]
\includegraphics[height=2.3 in]{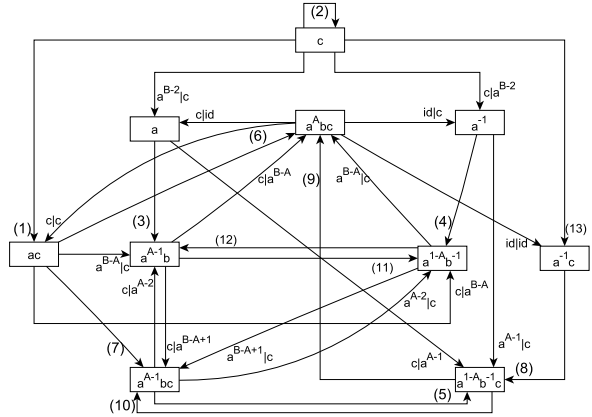}
\caption{The graph $G(\mathcal{S}'')$ for $A\geq 3$ and $B\geq 5$ and $2A<B+3$.}\label{Uni-graph}
\end{figure}
\begin{figure}[h] 
\subfigure[Proposition \ref{neigh_gra_S}, case $A=2,B=2$] { \includegraphics[width=2 in]{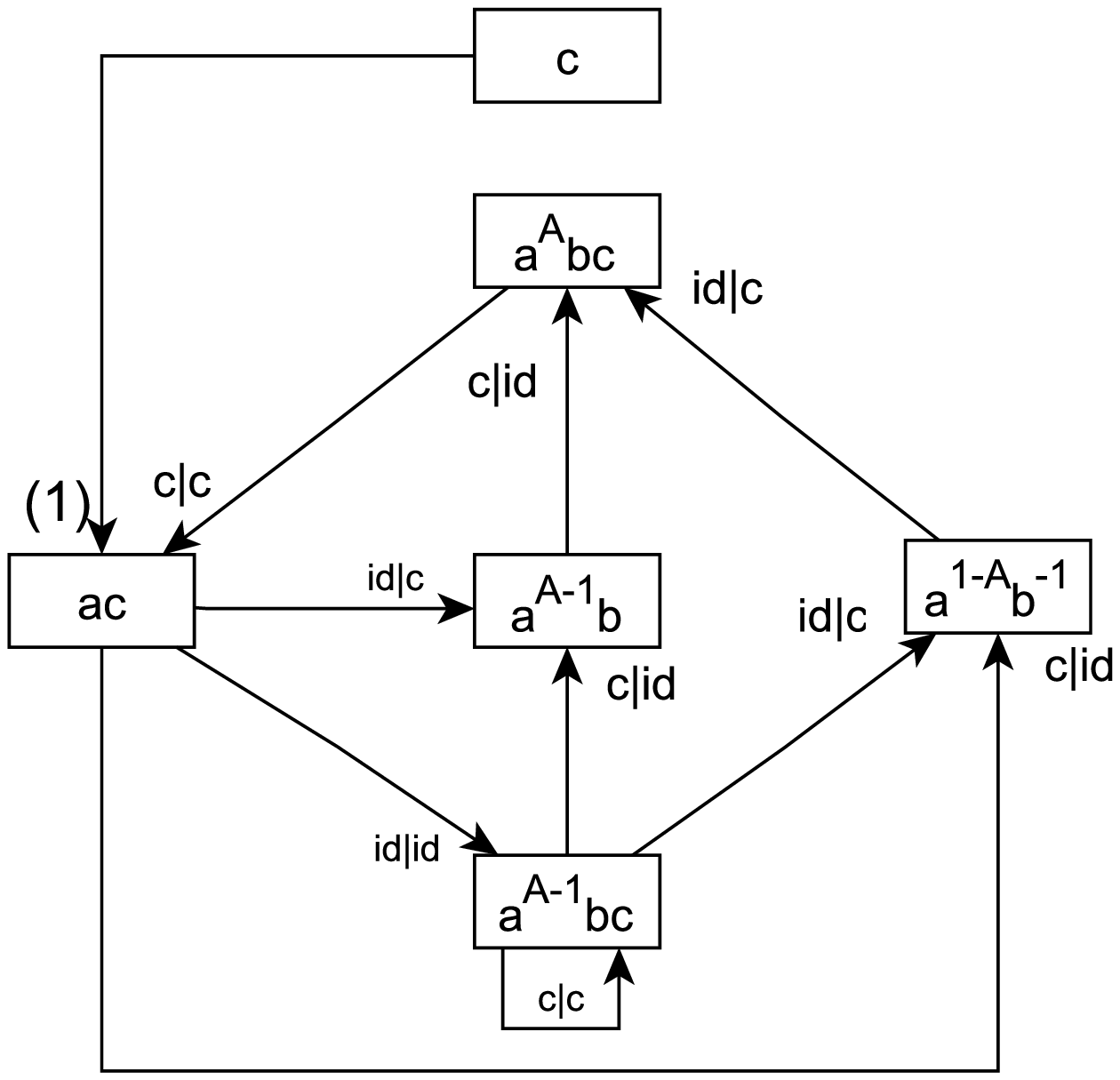}  \label{fig_secondsub} } \quad
\subfigure[Proposition \ref{neigh_gra_S}, case $A=2,B\geq3$]{ \includegraphics[width=2.5 in]{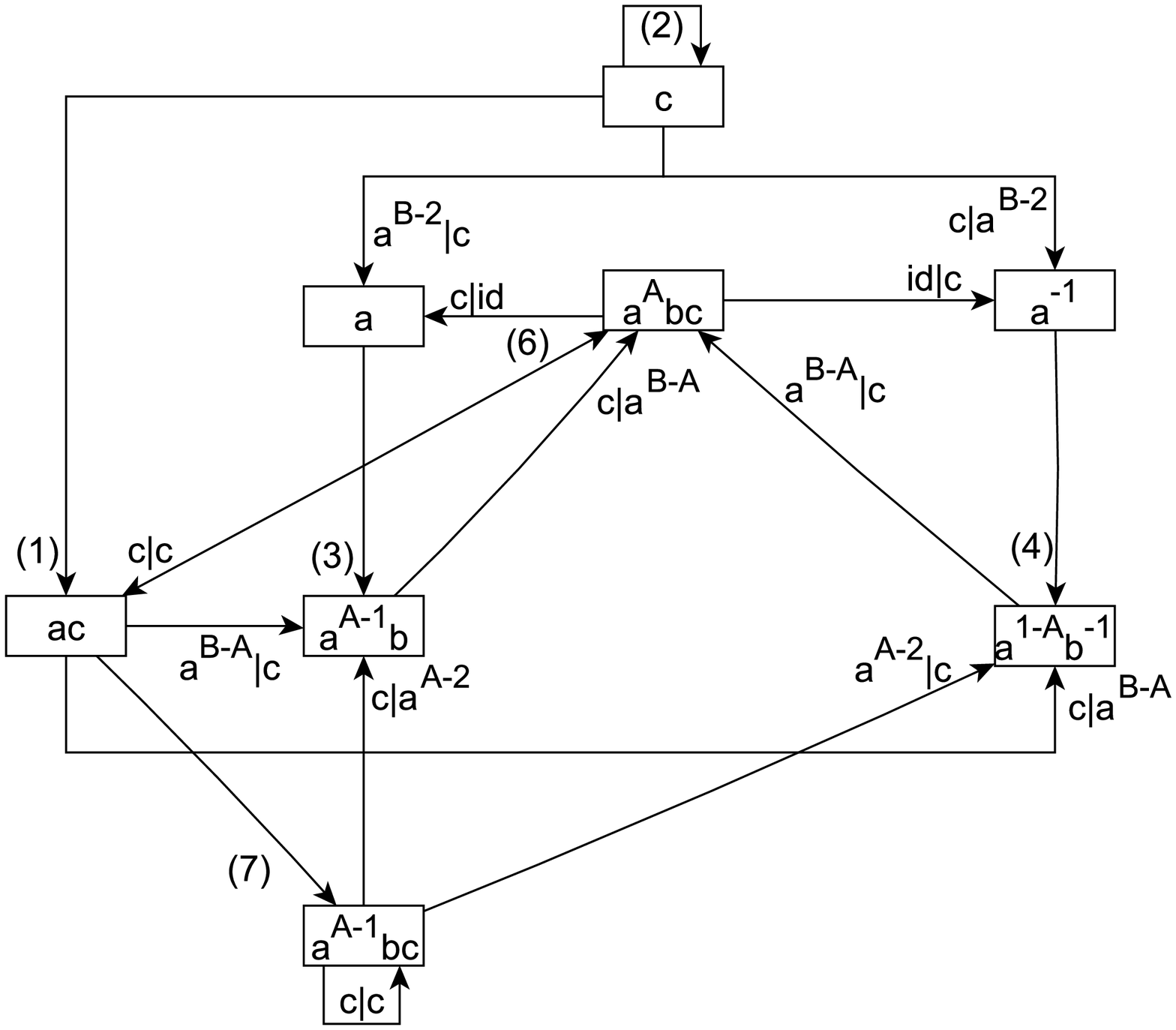} \label{fig_thirdsub} } 
\caption{The neighbor graph $G(\mathcal{S})$ of $T$}
\label{fig_subfigures2}

 \end{figure}

\begin{tabular}{|c|c|c|c|c|}
\hline 
\Small Edge & Labels & $\quad$ & Name & Condition \\ 
\hline
$c\rightarrow ac$ & 
$\quad\quad\begin{matrix}
\\ 
 a^{B-2}\\ 
a^{B-3} \\ 
\dots \\ 
id\\
\end{matrix}$ & 
$\begin{matrix}
\\
id \\ 
a\\ 
\dots \\ 
a^{B-2}\\
\end{matrix}\quad\quad$ & $(1)$ & $B\geq 2$ and $A\geq -1$\\ 
\hline  
$c\rightarrow a^{-1}c$ &
 $\quad\quad\begin{matrix}
\\
 a^{B-2}\\ 
 a^{B-3}\\ 
\dots \\
a^{2}\\
 \end{matrix} $& $\begin{matrix}
\\
 a^2\\ 
 a^3\\
\dots \\ 
a^{B-2}\\
 \end{matrix}\quad\quad$ & $(13)$ & $B\geq 4$ and  $A\geq-1$\\
\hline 
$c\rightarrow c$ &
 $\quad\quad\begin{matrix}
\\
 a^{B-2}\\ 
a^{B-3} \\ 
\dots \\
a\\
 \end{matrix} $& $\begin{matrix}
\\
 a\\ 
 a^2\\
\dots \\ 
a^{B-2}\\
 \end{matrix}\quad\quad$ & $(2)$ & $B\geq 3$ and $A\geq -1$\\
\hline
\vspace{-2mm}
 &&&&\\
 $c\rightarrow a^{-1}$& $c$&$a^{B-2}$&$\quad$&$B\geq 2,A\geq -1$ \\
  \hline
\vspace{-2mm}
 &&&&\\
 $c\rightarrow a$& $a^{B-2}$&$c$&$\quad$&$B\geq 2,A\geq -1$ \\

\hline 
$a\rightarrow a^{A-1}b$ &
 $\quad\quad\begin{matrix}
 \\
 id\\ 
a \\ 
\dots \\
a^{B-A-1}\\
 \end{matrix}$&$\begin{matrix}
\\
 a^{A-1}\\ 
 a^A\\
\dots \\ 
a^{B-2}\\
\end{matrix}\quad\quad$ & $(3)$ & 
 $\begin{matrix}
 B\geq 2,A\geq1\\
 \text{ and } (A,B)\neq(2,2)\\
\end{matrix}  $ \\
 \hline
 $a\rightarrow a^{1-A}b^{-1}c$ &$\quad\quad c$& $a^{A-1}\quad\quad$&$ \quad$&  $\begin{matrix}
 B\geq 2,A\geq1\\
 \text{ and } (A,B)\neq(2,2)\\
\end{matrix}  $ \\
 \hline
 $a\rightarrow bc$ &$\quad\quad id$& $c\quad\quad$&$\quad$ &  $B\geq 2,A\in\{-1,~0,~1\}$ \\
 \hline
 \vspace{-2mm}
 &&&&\\
 $a\rightarrow a^{-1}bc$ &$\quad\quad a$& $c\quad\quad$&$\quad$ &  $B\geq 3, A\in\{0,-1\}$ \\
 \hline
 \vspace{-2mm}
 &&&&\\
 $a^{-1}\rightarrow a^{-1}bc$ &$\quad\quad c$& $a\quad\quad$&$\quad$ &  $B\geq 3, A\in\{0,-1\}$ \\
 \hline
 \vspace{-2mm}
 &&&&\\
 $a^{-1}\rightarrow bc$ &$\quad\quad c$& $id\quad\quad$&$\quad$ &  $B\geq 3,A\in\{-1,~0,~1\}$\\
\hline
  \end{tabular}
  
  {\center{\quad\quad\quad\quad\bf{ Table 1. Edges of $G(\mathcal{S}'')$ (case $A\geq -1 \text{ and }2A<B+3$)}}

}
  
\begin{tabular}{|c|c|c|c|c|}
\hline 
Edge & Labels & $\quad$ & Name & Condition \\

\hline 
$a^{-1}\rightarrow a^{1-A}b^{-1}$ &
 $\quad\quad\begin{matrix}
 \\
 a^{A-1}\\ 
 a^{A}\\ 
\dots \\
a^{B-2}\\
 \end{matrix}$& $\begin{matrix}
 \\
 id\\ 
 a\\
\dots \\ 
a^{B-A+1}\\
 \end{matrix}\quad\quad$ & $(4)$ &$\begin{matrix}
 B\geq 2,A\geq1\\
 \text{ and } (A,B)\neq(2,2)\\
\end{matrix} $ \\
\hline
 $a^{-1}\rightarrow a^{1-A}b^{-1}c$ &$a^{A-1}$& $c\quad\quad$&\quad &  $\begin{matrix}
 B\geq 2,A\geq1\\
 \text{ and } (A,B)\neq(2,2)\\
\end{matrix}  $ \\
\hline
\vspace{-2mm}
 &&&&\\
$abc\rightarrow a^{-1}bc$&$id$&$id$&$\quad$ &$B\geq 2, A=0$\\
\hline 
$a^{A-1}bc\rightarrow a^{1-A}b^{-1}c$ &
 $\quad\quad\begin{matrix}
 \\
 a^{A-2}\\ 
 a^{A-3}\\ 
\dots \\
id\\
 \end{matrix} $& $\begin{matrix}
 \\
 id\\ 
 a\\
\dots \\ 
a^{A-2}\\
 \end{matrix}\quad\quad$ & $(5)$ &$B\geq2$ and $A\geq2$\\  
\hline
\vspace{-2mm}
 &&&&\\
$a^{A-1}bc\rightarrow abc$&$c$&$c$&$\quad$ &$B\geq 2,A\in\{-1,~0,~1\}$\\
  \hline 
  \vspace{-2mm}
 &&&&\\
  $a^{A-1}bc\rightarrow a^{A-1}b$&$c$&$a^{A-2}$&$\quad$ &$B\geq 2, A\geq2$\\
  \hline
  \vspace{-2mm}
 &&&&\\
    $a^{A-1}bc\rightarrow a^{A-1}b$&$c$&$a^{A-2}$&$\quad$ &$B\geq 2, A\geq2$\\
\hline
\vspace{-2mm}
 &&&&\\
  $a^{A-1}bc\rightarrow a^{1-A}b^{-1}$&$a^{A-2}$&$c$&$\quad$ &$B\geq 2, A\geq2$\\
   \hline
$ac\rightarrow a^{A}bc$ &
 $\quad\quad\begin{matrix}
 \\
 a^{B-A-1}\\ 
 a^{B-A-2}\\ 
\dots \\
id\\
 \end{matrix} $& $\begin{matrix}
 \\
 id\\ 
 a\\
\dots \\ 
a^{B-A+1}\\
 \end{matrix}\quad\quad$ & $(6)$ & $\begin{matrix}
 B\geq 2,A\geq1\\
 \text{ and } (A,B)\neq(2,2)\\
\end{matrix}  $ \\
   \hline
$ac\rightarrow a^{-1}bc$ &
 $\quad\quad\begin{matrix}
 \\
 a^{B-2}\\ 
 a^{B-3}\\ 
\dots \\
id\\
 \end{matrix} $& $\begin{matrix}
 \\
 a^2\\ 
 a^3\\
\dots \\ 
a^{B-2}\\
 \end{matrix}\quad\quad$ & $(6)'$ & $
 B\geq 4,A\in\{0,-1\}$ \\
    \hline
$ac\rightarrow abc$ &
 $\quad\quad\begin{matrix}
 \\
 a^{B-2}\\ 
 a^{B-3}\\ 
\dots \\
id\\
 \end{matrix} $& $\begin{matrix}
 \\
 id\\ 
 a\\
\dots \\ 
a^{B-2}\\
 \end{matrix}\quad\quad$ & $(14)$ & $
 B\geq 2, A=0$ \\
\hline 
$ac\rightarrow a^{A-1}bc$ &
 $\quad\quad\begin{matrix}
 \\
 a^{B-A}\\ 
 a^{B-A-1}\\ 
\dots \\
id\\
 \end{matrix} $& $\begin{matrix}
 \\
 id\\ 
 a\\
\dots \\ 
a^{B-A}\\
 \end{matrix}\quad\quad$ & $(7)$ & $B\geq 2$ and  $A\geq 2$\\

  \hline
    $ac\rightarrow bc$&$
  \quad\quad\begin{matrix}
  \\
  a^{B-2}\\
  a^{B-3}\\
  \dots\\
  a\\
  
\end{matrix} $&$
\begin{matrix}
\\
a\\
a^2\\
\dots\\
a^{B-2}\\
\end{matrix}\quad\quad$&$(7)'$&$B\geq 3, A\in\{-1,~0,~1\}$\\
\hline  
\vspace{-2mm}
 &&&&\\
$ac\rightarrow a^{A-1}b$&$a^{B-A}$&$c$&$\quad$ &$B\geq 2, A\geq 2$\\
  \hline
  \vspace{-2mm}
 &&&&\\  
  $ac\rightarrow a^{1-A}b^{-1}$&$c$&$a^{B-A}$&$\quad$ &$B\geq 2, A\geq2$\\
\hline
\vspace{-2mm}
 &&&&\\
  $a^Abc\rightarrow a^{-1}c$&$id$&$id$&$\quad$ &$B\geq 2, A\geq-1$\\

\hline 
\end{tabular} 
 {\center{\quad\quad\bf{Table 2. Edges of $G(\mathcal{S}'')$ (Case $A\geq -1 \text{ and }2A<B+3$)}}
  
}

\begin{tabular}{|c|c|c|c|c|}
\hline 
Edge & Labels & $\quad$ & Name & Condition \\

  \hline
  \vspace{-2mm}
 &&&&\\
  $a^Abc\rightarrow a$&$c$&$id$&$\quad$ &$B\geq 2, A\geq -1$\\
  \hline
  \vspace{-2mm}
 &&&&\\
  $a^Abc\rightarrow a^{-1}$&$id$&$c$&$\quad$ &$B\geq 2, A\geq -1$\\
  \hline
  \vspace{-2mm}
 &&&&\\
  $a^Abc\rightarrow ac$&$c$&$c$&$\quad$ &$B\geq 2, A\geq -1$\\
  \hline
  \vspace{-2mm}
 &&&&\\
  $bc\rightarrow a^{-1}bc$&$id$&$id$&$\quad$ &$B\geq 2, A=-1$\\
 \hline
$a^{-1}c\rightarrow a^{1-A}b^{-1}c$ &
 $\quad\quad\begin{matrix}
 \\
 a^{B-2}\\ 
 a^{B-3}\\ 
\dots \\
a^{A}\\
 \end{matrix} $& $\begin{matrix}
 \\
 a^A\\ 
 a^{A+1}\\
\dots \\ 
a^{B-2}\\
 \end{matrix}\quad\quad$ & $(8)$ & $B\geq A+2, A>0$\\  
\hline
\vspace{-1.5mm}
 &&&&\\
  $a^{-1}c\rightarrow a^{-1}bc$&$c$&$c$&$\quad$ &$B=2,~A=-1$
  \\
\hline 
$a^{1-A}b^{-1}c\rightarrow a^{A}bc$ &
 $\quad\quad\begin{matrix}
 \\
 a^{B-2}\\ 
 a^{B-3}\\ 
\dots \\
a^{B-A+1}\\
 \end{matrix} $& $\begin{matrix}
 \\
 a^{B-A+1}\\ 
 a^{B-A+2}\\
\dots \\ 
a^{B-2}\\
 \end{matrix}\quad\quad$ & $(9)$ & $B\geq 4$ and  $A\geq3$\\ 
    \hline
$a^{1-A}b^{-1}c\rightarrow a^{A-1}bc$ &
 $\quad\quad\begin{matrix}
 \\
 a^{B-2}\\ 
 a^{B-3}\\ 
\dots \\
a^{B-A+2}\\
 \end{matrix} $& $\begin{matrix}
 \\
 a^{B-A+2}\\ 
 a^{B-A+3}\\
\dots \\ 
a^{B-2}\\
 \end{matrix}\quad\quad$ & $(10)$ & $B\geq 6$ and  $A\geq4$\\ 

\hline
$a^{A-1}b\rightarrow a^{1-A}b^{-1}$ &
 $\quad\quad\begin{matrix}
 \\
 id\\ 
 a\\ 
\dots \\
a^{A-3}\\
 \end{matrix} $& $\begin{matrix}
 \\
 a^{B-A+1}\\ 
 a^{B-A+2}\\
\dots \\ 
a^{B-2}\\
 \end{matrix}\quad\quad$ & $(11)$ & $B\geq 4$ and  $A\geq3$\\ 
 \hline
 \vspace{-2mm}
 &&&&\\
 $a^{A-1}b\rightarrow a^{A}bc$ &$c$& $a^{B-A}$&$\quad $& $B\geq 2$ and $A\geq 2$\\
 \hline
 \vspace{-2mm}
 &&&&\\
 $a^{A-1}b\rightarrow a^{A-1}bc$ &$c$& $a^{B-A+1}$&\quad & $B\geq 4$ and $A\geq 3$\\
\hline 
$a^{1-A}b^{-1}\rightarrow a^{A-1}b$ &
 $\quad\quad\begin{matrix}
 \\
 a^{B-A+1}\\ 
 a^{B-A+2}\\ 
\dots \\
a^{B-2}\\
 \end{matrix} $& $\begin{matrix}
 \\
 id\\ 
 a\\
\dots \\ 
a^{A-3}\\
 \end{matrix}\quad\quad$ & $(12)$ & $B\geq 4$ and  $A\geq3$\\ 
\hline
\vspace{-2mm}
 &&&&\\
 $a^{1-A}b^{-1}\rightarrow a^{A}bc$ &$a^{B-A}$& $c$&\quad & $B\geq 2$ and $A\geq 2$\\
 \hline
\vspace{-2mm}
 &&&&\\
 $a^{1-A}b^{-1}\rightarrow a^{A-1}bc$ &$a^{B-A+1}$& $c$&\quad & $B\geq 4$ and $A\geq 3$\\ 
\hline 
\end{tabular} 
{\center{\bf{Table 3. Edges of $G(\mathcal{S}'')$ (Case $A\geq -1 \text{ and }2A<B+3$)}}
  
  }
\medskip  
  
Since $\mathcal{S}\subset\mathcal{S}''$, it is clear that the neighbor graph $G(\mathcal{S})$ is a subgraph of the pseudo-neighbor graph.  We will see that Characterization \ref{B_r} will play an important role in the relationship between the neighbor graph $G(\mathcal{S})$ and the pseudo-neighbor graph $G(\mathcal{S}'')$.

 \begin{figure}[h]
 \includegraphics[width=3 in]{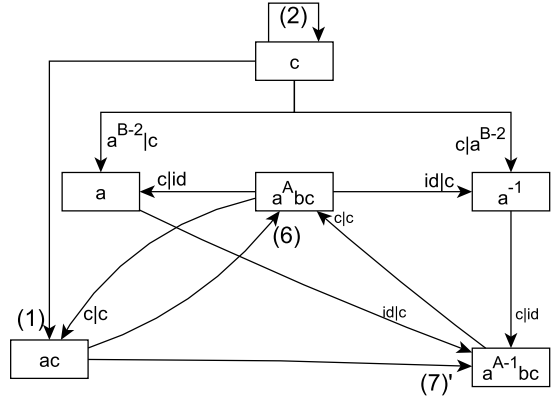} 
 \caption{Proposition \ref{neigh_gra_S}, Case $A=1,B\geq2$. We refer to Tables 1, 2 and 3 for the conditions on the edges.} \label{fig_subfigures1}
 \end{figure}
\begin{theorem}\label{neigh_gra_S}
Let $\mathcal{S}$ be the neighbor set of $T$ and $\mathcal{S}''$ be defined as in \eqref{Pseudo-neigh}, \eqref{Pseudo-neigh0} and \eqref{Pseudo-neigh-1}. The following results hold for $A, B$ satisfying $-1\leq A\leq B,~B\geq 2$ and $2A<B+3$.
\begin{enumerate}
\item For $A\geq 3$ and $B\geq 5$,   $\mathcal{S}=\mathcal{S}''$, that is, $$\mathcal{S}=\{a,~a^{-1},~a^{A-1}b,~a^{1-A}b^{-1},~c,~ac,~a^{-1}c,~a^{A-1}bc,~a^{1-A}b^{-1}c,~a^Abc\}.$$ \label{neigh_gra_S1}
\item For $A=3$ and $B=4$, $\mathcal{S}=\mathcal{S}''\setminus \{a^{-1}c\}$, that is, $$\mathcal{S}=\{a, ~a^{-1},~a^2b, ~a^{-2}b^{-1}, ~c, ~ac, ~a^2bc, ~a^{-2}b^{-1}c, ~a^3bc\}.$$ \label{neigh_gra_S2}
\item For $A=2$ and $B=2$, $\mathcal{S}=\mathcal{S}''\setminus \{a^{-1}c, a^{1-A}b^{-1}c, a, a^{-1}\}$, that is, 
$$\mathcal{S}=\{ab, ~a^{-1}b^{-1}, ~c,~ac,  ~abc, ~a^2bc\}.$$ \label{neigh_gra_S3}
\item For $A=2$ and $B\geq3$, $\mathcal{S}=\mathcal{S}''\setminus \{a^{-1}c, a^{1-A}b^{-1}c\}$, that is, 
$$\mathcal{S}=\{a, ~a^{-1}, ~ab, ~a^{-1}b^{-1}, ~c, ~ac, ~abc, ~a^2bc\}.$$\label{neigh_gra_S4}
\item For $A=1$ and $B\geq 2$, $\mathcal{S}=\mathcal{S}''\setminus\{a^{-1}c,a^{1-A}b^{-1}c, a^{A-1}b, a^{1-A}b^{-1}\}$, that is, 
$$\mathcal{S}=\{a, ~a^{-1}, ~c, ~ac, ~abc, ~bc\}.$$ \label{neigh_gra_S5}
\item For $A=0$ and $B\geq 2$, 
$$\mathcal{S}=\{a, ~a^{-1}, ~c, ~ac, ~a^{-1}bc, ~bc, abc\}.$$ \label{neigh_gra_S6}
\item For $A=-1$ and $B=2$, 
$$\mathcal{S}=\{a, ~a^{-1}, ~c, ~a^{-1}c, ~a^{-1}bc, ~bc\};$$
 For $A=-1$ and $B\geq 3$, 
$$\mathcal{S}=\{a, ~a^{-1}, ~c, ~ac, ~a^{-1}bc, ~bc\}.$$ \label{neigh_gra_S7}

\end{enumerate}
\end{theorem}

\begin{proof}
By Characterization \ref{B_r}, the neighbor graph $G(\mathcal{S})$ is obtained from the pseudo-neighbor graph $G(\mathcal{S}'')$ by deleting the states that are not the starting state of an infinite walk.
For $A\geq 3, B\geq 5$, from Figure \ref{Uni-graph}, it is clear that there is an infinite walk starting from each state of $G(\mathcal{S}'')$.
For $A=3, B=4$, from Table~ 1, Table~ 2 and Table~ 3, we know that there is exactly one state $a^{-1}c$ from which there is no outgoing edge.
For Item \eqref{neigh_gra_S3},\eqref{neigh_gra_S4},\eqref{neigh_gra_S5}, and \eqref{neigh_gra_S6}, see Figure~ \ref{fig_secondsub}, Figure~ \ref{fig_thirdsub}, Figure~ \ref{fig_subfigures1} and Figure~ \ref{Aequal0}, respectively.
For Item \eqref{neigh_gra_S7},  it is easy to check that $a^{-1}c$ is the starting state of an infinite walk if and only if $B=2$ and $ac$ is the starting state of an infinite walk if and only if $B\geq 3$. Since the neighbor set has at least six elements by Proposition \ref{num_p2}, we get the results of Item~(\ref{neigh_gra_S7}) (see Figure \ref{Aequal-1} for more details). 
\end{proof}

\begin{figure}[h] 
\includegraphics[width=3 in]{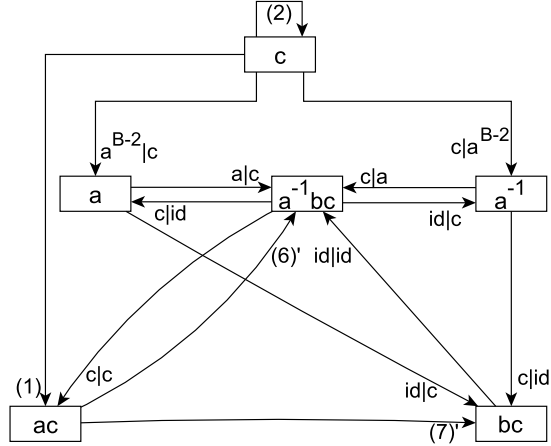}  
\caption{Proposition \ref{neigh_gra_S}, case $A=-1,B\geq3$. For the case $B=2$, we only need to replace $ac$ by $a^{-1}c$ and change the incoming and outgoing edges according to Tables 1, 2 and 3.}\label{Aequal-1}
\end{figure}
\begin{figure}[h]
 \includegraphics[width=3 in]{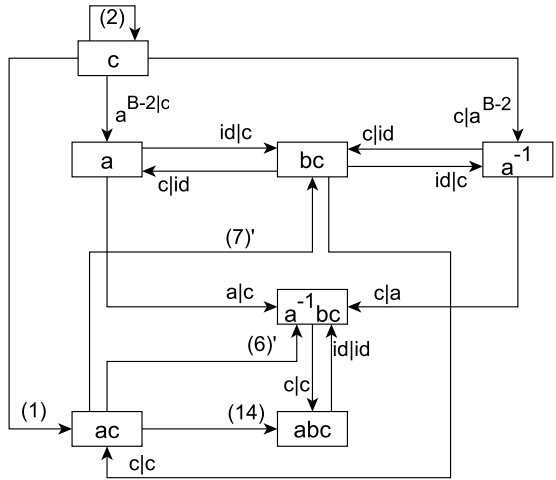}
 \caption{Proposition \ref{neigh_gra_S}, the case $A=0,B\geq2$. We refer to Tables 1, 2 and 3 for the conditions on the edges.}\label{Aequal0}  
 \end{figure}



\section{Characterization of the disk-like tiles for $A\geq -1$ and $2A<B+3$}\label{sec_disklikepos}

We are now in a position to study the topological properties of our family of $p2$-tiles under the conditions $A\geq -1$, $2A<B+3$. We will characterize the disk-like tiles of the family under this condition. Loridant and Luo in \cite{LoridantLuo09} provided necessary and sufficient conditions for a $p2$-tile to be disk-like.  Before stating the theorem, we need a definition.
\begin{definition}(\cite{LoridantLuo09})
If $\mathcal{P}$ and $\mathcal{F}$ are two sets of isometries in $\mathbb{R}^2$, we say that $\mathcal{P}$ is $\mathcal{F}$-connected iff for every disjoint pair $(d,d')$ of elements in $\mathcal{P}$, there exist $n\geq 1$ and elements $d=:d_0, d_1,\dots, d_{n-1},d_n:=d'$ of $\mathcal{P}$ such that $d_i^{-1}d_{i+1}\in\mathcal{F}$ for $i=0,1,\dots,n-1.$
\end{definition}

The following statement is from \cite{LoridantLuo09}.  In fact, the necessary part is due to the classification of Gr\"unbaum and Shephard \cite{GruenbaumShephard89}.

\begin{proposition}\label{disk_rule}
Let $K$ be a crystile that tiles the plane by a $p2$-group. Let $\mathcal{F}$ be the corresponding digit set.
\begin{enumerate}
\item Suppose that the neighbor set $\mathcal{S}$ of $K$ has six elements. Then $K$ is disk-like iff $\mathcal{F}$ is $\mathcal{S}$-connected.\label{p1}
\item Suppose that the neighbor set $\mathcal{S}$ of $K$ has seven elements 
$$\{b^{\pm 1},c,bc,a^{-1}c,a^{-1}bc,a^{-1}b^{-1}c\},$$ 
where $a,b$ are translations, and $c$ is a $\pi$-rotation. Then $K$ is disk-like iff $\mathcal{F}$ is $\{b^{\pm 1},c,bc,a^{-1}c\}$-connected. \label{p2}
\item Suppose that the neighbor set $\mathcal{S}$ of $K$ has eight elements 
$$\{b^{\pm 1},c,bc,a^{-1}c,b^{-1}c,a^{-1}bc, a^{-1}b^{-1}c\}$$
$$(\text{resp. } \{b^{\pm 1},(a^{-1}b)^{\pm 1},c,bc,ac, ab^{-1}c,a^{-1}bc\}),$$
where $a,b$ are translations, and $c$ is a $\pi$-rotation. Then $K$ is disk-like iff $\mathcal{F}$ is $\{b^{\pm 1},c,a^{-1}c\}$-$(\text{resp. } \{c,bc,ac, ab^{-1}c\}-)$connected. \label{p3}
\item Suppose that the neighbor set $\mathcal{S}$ of $K$ has twelve elements 
$$\{a^{\pm 1},b^{\pm 1},(ab)^{\pm 1},c,a^{-1}c,bc,abc,a^{-1}bc,a^{-1}b^{-1}c\},$$
where $a,b$ are translations, and $c$ is a $\pi$-rotation. Then $K$ is disk-like iff $\mathcal{F}$ is $\{c,a^{-1}c,bc\}$-connected.\label{p4}
\end{enumerate} 

\end{proposition}

Applying this result, we obtain the following theorem.
\begin{theorem}\label{diskCrystile}
Let $A,B\in\mathbb{Z}$ satisfy $-1\leq A\leq B,~ B\geq 2$ and $2A<B+3$, and let $T$ be the crystallographic replication tile defined by the data $(g,\mathcal{D})$ given in~\eqref{expand-map} and \eqref{digit}. Then the following statements hold.
\begin{enumerate}
\item If $A\in\{-1,~0,~1\}, B\geq 2$ or $A=2, ~B=2$, then $T$ is disk-like. \label{M1}
\item If $A\geq 2, B\geq3$, then $T$ is non-disk-like. \label{M2}

\end{enumerate}
\end{theorem}

\begin{proof}
Let $\mathcal{S}$ be the neighbor set of $T$.
By  Theorem \ref{neigh_gra_S}, we know that in the assumption of $A\in \{-1, ~1\},~B\geq 2$ and $A=2, B=2$, the neighbor sets of $T$ all have six elements. Let us check the case $A=1, B\geq 2$ by showing that $\mathcal{D}$ is $\mathcal{S}$-connected and applying Proposition \ref{disk_rule} \eqref{p1}. Then $A=-1, B\geq 2$ and $A=2, B=2$ can be checked in the same way.

For $A=1, B\geq 2$, the digit set is $\mathcal{D}=\{id, a,\dots, a^{B-2}, c\}$ and the neighbor set is $\mathcal{S}=\{a, a^{-1}, c, abc, bc, ac\}$. It is easy to find that the disjoint pairs $(d,d')$ in $\mathcal{D}\times \mathcal{D}$ are the following ones:
\begin{equation}\label{pairs}
(id,a^\ell), (a^\ell ,id),
(id,c), (c ,id),
(a^k,a^{k'})
(a^j,c), \text{or }(c,a^j),
\end{equation}
where $\ell,k, k',j \in\{1,2,\dots,B-2 \}$.

We will check the pair $(a^k,a^{k'})$ at first.
If $k<k'$,  then let $n=k'-k$, and 
$$d_0=a^k, d_1=a^{k+1},\dots,d_{n-1}=a^{k'-1},d_n=a^{k'}.$$
hence  $d_i^{-1}d_{i+1}=a$ is in $\mathcal{S}$ for $0 \leq i\leq n-1$.
If $k>k'$,  $d_i^{-1}d_{i-1}=a^{-1}$ is also in $\mathcal{S}$ for $0\leq i \leq n-1$.
To check $(id, a^\ell)$ and $(a^j, c$), it suffices to check $(id, a)$ and $(a,c)$. It is clear for $(id,a)$. For $(a,c)$, let $n=2$, and $d_0=a, ~d_1=id, ~d_2=c$. Hence, we have proved that $\mathcal{D}$ is $\mathcal{S}$- connected. By Proposition \ref{disk_rule} \eqref{p1}, $T$ is disk-like.

For $A=0 \text{ and } B\geq 2$ and the neighbor set $$\mathcal{S}=\{a, a^{-1}, c, a^{-1}bc, bc, ac, abc\}$$ has seven elements. By Proposition \ref{disk_rule} \eqref{p2}, we need to prove that $\mathcal{D}$ is 
$\{a, a^{-1}, c,\\ ac, bc\}$-connected. This is achieved in the same way as above.

We now prove Item \eqref{M2}. For $A=2, B\geq 3$ and by Theorem \ref{neigh_gra_S}, we know that
$$\mathcal{S}=\{a, ~a^{-1}, ~ab, ~a^{-1}b^{-1}, ~c, ~abc, ~a^2bc, ~ac\}.$$
Let $a'=a^2b, b'=ab$, then $\mathcal{S}$ has the form 
$$\Upsilon:=\{b', ~b'^{-1}, ~a'^{-1}b',  ~a'b'^{-1}, ~c, ~b'c, ~a'c, ~a'b'^{-1}c \}$$
of Proposition \ref{disk_rule} \eqref{p3}. However, it is easily checked that $\mathcal{D}$ is not 
$\{c, abc, ab^2c,ac\}$-connected. By Proposition \ref{disk_rule} \eqref{p3}, $T$ is not disk-like.

For $A\geq 3, ~B\geq 4$,  we have 
$\sharp\mathcal{S}=9$ if $A=3, B=4$, and $\sharp\mathcal{S}=10$ if $A\geq 3, B\geq 5$ by Theorem  \ref{neigh_gra_S}. According to Gr\"unbaum and Shephard's classification of isohedral tillings (see \cite[Sect. 6.2, p.285]{GruenbaumShephard89}), the cases in Proposition \ref{disk_rule} are the only ones leading to disk-like $p2$-tiles in the plane. So  $T$ is non-disk-like for $A\geq 3, B\geq4$.
\end{proof}

\section{Characterization of the disk-like tiles for $A\leq -2$ and $2|A|<B+3$}\label{sec_disklikeneg}
We now deal with the case $A\leq -2$ and $2|A|<B+3$. Let us recall a statement in \cite[Equation (2.11), p. 2177]{AkiyamaLoridant10}. Let $T^{\ell}$ be the lattice tile associated with the expanding matrix
 $M=\left(\begin{matrix}
0 & -B\\
1 & -A\\
\end{matrix}\right)$ and digit set $\mathcal{D}$ (see \eqref{digit}) and $\bar{T}^{\ell}$ the lattice tile associated with the matrix  $\Bar{M}=\left(\begin{matrix}
0 & -B\\
1 & A\\
\end{matrix}\right)$ and $\mathcal{D}$. Then we have 
\begin{equation}\label{relat_pos_neg}
\bar{T}^{\ell}=\left(\begin{matrix}
-1 & 0\\
0 & 1\\
\end{matrix}\right)T^{\ell}+\sum_{k=1}^{\infty}\bar{M}^{-2k}\left(\begin{matrix}
B-1 \\
0\\
\end{matrix}\right).
\end{equation}
It follows that $T^{\ell}$ and $\bar{T}^{\ell}$ have the same topology. It is remarkable that this does not hold for the associated crystiles $T$ and $\bar{T}$, as is illustrated below. 

By  \cite{AkiyamaThuswaldner05}, we know all the information on the neighbor set of the lattice tile $T^{\ell}$  for $A\geq -1$, hence we can derive the neighbor set of $\bar{T}^{\ell}$ immediately.

\begin{lemma}\label{lem:aff} 
If $2A<B+3$ and $A>0$, then the neighbor set of $\bar{T}^{\ell}$ is
\begin{equation}\label{neigh_nega_latt}
\{(-A,1), (-A+1,1), (-1,0), (1,0), (A,-1), (A-1,-1)\},
\end{equation}
or, using translation mappings  rather than vectors, 
\begin{equation}
\{a^{-A}b, a^{-A+1}b, a^{-1}, a, a^Ab^{-1}, a^{A-1}b^{-1}\}.
\end{equation} 
\end{lemma}
\begin{proof}
The vector $\gamma=\left(\begin{matrix}
p\\
q\\
\end{matrix}\right)
\in\mathbb{Z}^2$ is a neighbor of $T^{\ell}$ iff $T^{\ell}\cap (T^{\ell}+\gamma)\neq\emptyset$. Let $\gamma'=\left(\begin{matrix}
-1 & 0\\
0 & 1\\
\end{matrix}\right)\left(\begin{matrix}
p \\
q \\
\end{matrix}\right)$, then this is equivalent to $\bar{T}^{\ell}\cap (\bar{T}^{\ell}+\gamma')\neq\emptyset$ by \eqref{relat_pos_neg}. Thus, using Proposition \ref{neigh_lattice}, we get the neighbor set \eqref{neigh_nega_latt} of $\bar{T}^{\ell}$.
\end{proof}

For $-1\leq A\leq B\geq 2$, the data $(g,\mathcal{D},p2)$ is a crystallographic number system, hence, the tiling group is the whole crystallographic group $p2$~\cite{Loridant12}. It follows from Lemma~\ref{lem:tiling} that this property still holds for $A\leq -2$. Now, by Lemma~\ref{lem:aff}, to obtain the neighbor set of $p2$-crystiles for $ A \leq -2$, we only need to repeat the methods in Section~\ref{sec_neigh1} and~\ref{sec_neigh2}, dealing with  similar  estimates and computations. We come to the following theorem for $A\leq-2$ (we do not reproduce the computations).
\begin{theorem}\label{Minius}
Let $A,B\in\mathbb{Z}$ satisfy $2\leq -A\leq B$ and $2|A|<B+3$, and let $T$ be the crystallographic replication tile defined by the data $(g,\mathcal{D})$ given in~\eqref{expand-map} and \eqref{digit}. Then the following statements hold.
\begin{itemize}
\item[(1)] For $A=-2$ and $B=2 \text{ or } 3$, the neighbor set of the crystile $T$ is 
$$\mathcal{S}=\{a, a^{-1}, c, a^{-1}c, a^{-2}bc, a^{-1}bc\};$$
\item[(2)] For $A=-2, B\geq 4$, the neighbor set of the crystile $T$ is 
$$\mathcal{S}=\{a,~ a^{-1},~ c,~ ac, ~a^{-2}bc, ~a^{-1}bc\}.$$
\item[(3)] For $A=-3, B=4$, the neighbor set of the crystile $T$ is $$\mathcal{S}=\{a, ~ a^{-1}, ~a^{-2}b, ~a^2b^{-1}, ~ c,~ a^{-1}c, ~a^{-2}bc, ~a^{-3}bc\}.$$
\item[(4)] For $A=-3, B\geq 5$, the neighbor set of the crystile $T$ is $$\mathcal{S}=\{a, ~ a^{-1}, ~a^{A+1}b, ~a^{-1-A}b^{-1}, ~ c,~ ac,~ a^{A+1}bc, ~a^{A}bc, ~a^{-1}c\}.$$
\item[(5)] For $A=-4, B\geq 6$, the neighbor set of the crystile $T$ is $$\mathcal{S}=\{a, ~ a^{-1}, ~a^{A+1}b, ~a^{-1-A}b^{-1}, ~ c, ~a^{-1}c,~ ac, ~a^{A+1}bc, ~a^{A}~bc, ~a^{-A-1}b^{-1}c\}.$$
\end{itemize}
\end{theorem} 
Consequently, we can infer from Lemma \ref{disk_rule}  the following theorem.

\begin{theorem}\label{Minius1}
Let $A,B\in\mathbb{Z}$ satisfy $2\leq -A\leq B$ and $2|A|<B+3$, and let $T$ be the crystallographic replication tile defined by the data $(g,\mathcal{D})$ given in~\eqref{expand-map} and \eqref{digit}. Then the following statements hold.
\begin{enumerate}
\item If $A=-2, B\geq  2$, then $T$ is disk-like. \label{Minius1_1}
\item  If $A\leq -3, B\geq 4$, then $T$ is not disk-like.\label{Minius1_2}
\end{enumerate}

\end{theorem}

\begin{proof}
For Item \eqref{Minius1_1}, we know from Theorem \ref{Minius} that the neighbor set of $T$ has six neighbors. Thus, by Proposition \ref{disk_rule} Item \eqref{p1}, $T$ is disk-like.

For $A=-3, B=4$, the neighbor set is
$$\mathcal{S}=\{a, ~ a^{-1}, ~a^{-2}b, ~a^2b^{-1}, ~ c, ~a^{-2}bc, ~a^{-3}bc,~ a^{-1}c\}.$$
Let $a'=a^{-3}b, b'=a^{-1}$, then $\mathcal{S}$ has the form 
$$\Upsilon:=\{b', ~b'^{-1}, ~a'^{-1}b',  ~a'b'^{-1}, ~c, ~b'c, ~a'c, ~a'b'^{-1}c \}$$
of Proposition \ref{disk_rule} \eqref{p3}. However, it is easily checked that $\mathcal{D}$ is not 
$\{c,~ a^{-2}bc, ~ab^{-3}c, \\ ~a^{-1}c\}$-connected. By Proposition \ref{disk_rule} Item \eqref{p3}, $T$ is not disk-like.

For the cases $A=-3,~B\geq 5$ and $A\leq -4,~ B\geq 6$, $T$ has $9$ and $10$ neighbours, respectively. Thus $T$ is not disk-like as we have discussed in Theorem \ref{diskCrystile}. 
\end{proof}
In particular, we see that for $A=2$ and $B=3$, the crystile is not disk-like (from Theorem \ref{diskCrystile}), while for $A=-2$ and $B=3$, it is disk-like (see Figure \ref{fig:AMA}).

\begin{figure}
\centering
\includegraphics[width=2.3 in]{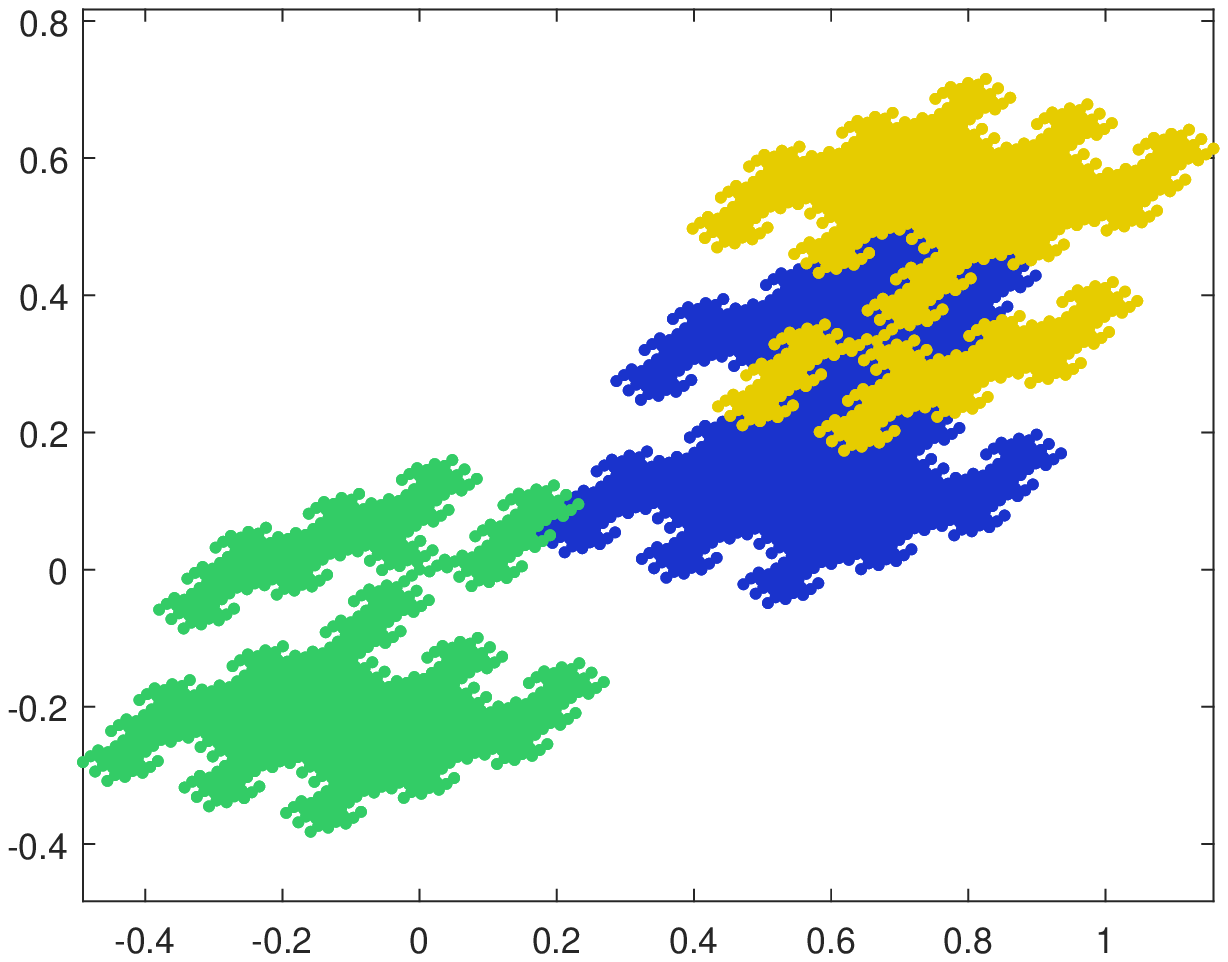}\quad \includegraphics[width=2.3 in]{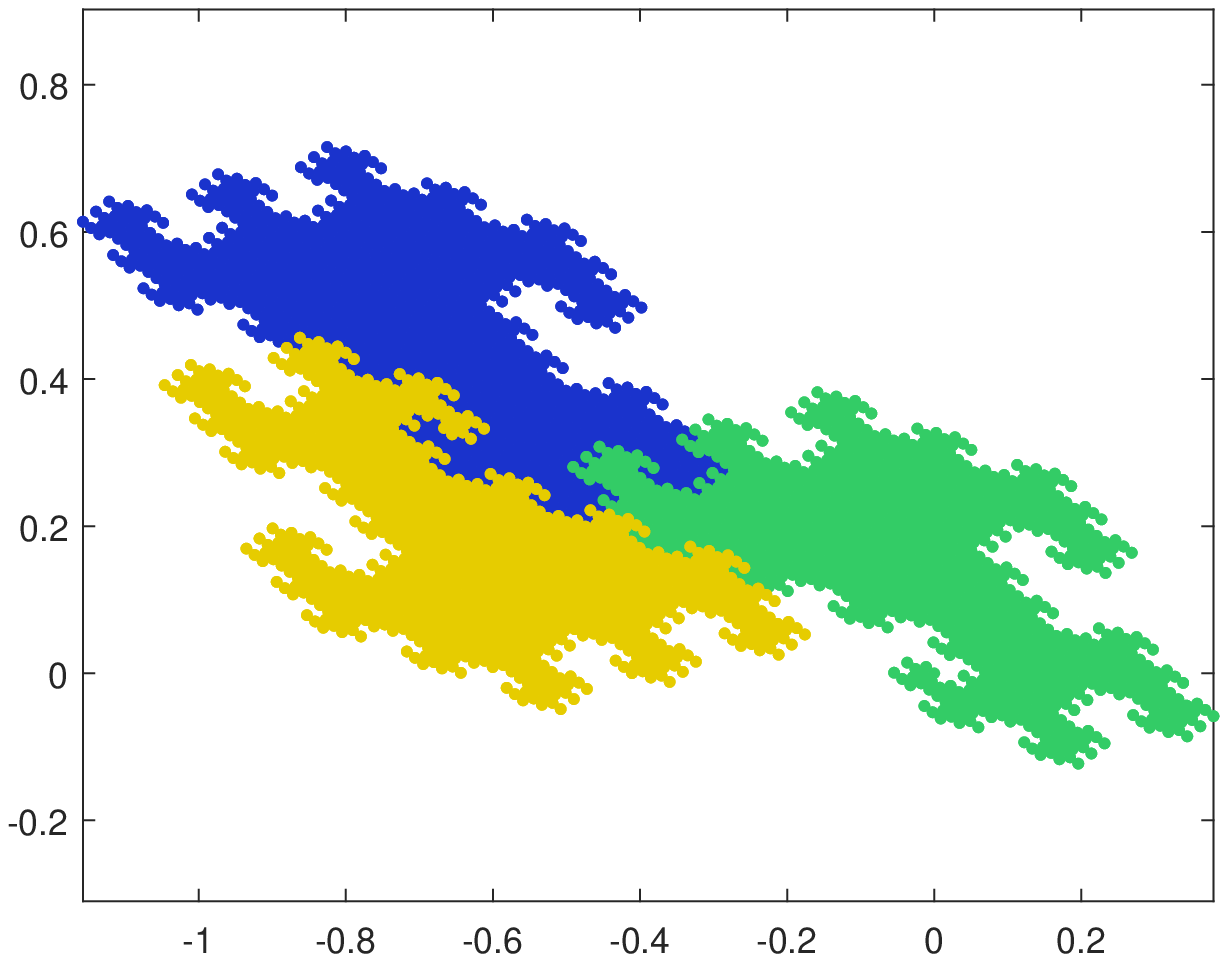}
\caption{$T$ for $A=2,B=3$ on the left and for $A=-2,B=3$ on the right}\label{fig:AMA}
\end{figure}

\section{Non-disk-likeness of tiles for $2|A|\geq B+3$}
\label{sec_nondisk}
So far, we have dealt with the case $2|A|<B+3$ and characterized the disk-like $p2$-tiles in Theorem \ref{diskCrystile} and Theorem \ref{Minius1}. If $2|A|\geq B+3$, it was proved in \cite{LauLeung07} that the lattice tiles $T^{\ell}$ are not disk-like. We prove that this also holds for the corresponding $p2$-tiles $T$. 

Recall that the $p2$-tile $T$ satisfies the equation
\begin{equation}\label{setequation}
T=\bigcup_{i=1}^B f_i(T),
\end{equation}
where 
$$f_1=g^{-1}\circ id,~ f_i=g^{-1}\circ a^{i-1} ~(2\leq i\leq B-1),~ f_B=g^{-1}\circ c,$$
$g$  is the expanding map,  
and $\mathcal{D}$ is the digit set  defined as before.
We denote the fixed point of a mapping $f$ by Fix$(f)$ and the linear part of $g$ by $M$. Then we have the following facts:
\begin{equation}\label{Fixpoint1}
\text{Fix}(f_i)=(M-I_2)^{-1}\left(\begin{matrix}
i-1-\frac{B-1}{2}\\
0\\
\end{matrix}\right) \text{ for } 1\leq i\leq B-1,
\end{equation}
\begin{equation}\label{Fixpoint2}
\text{Fix}(f_B)=(M+I_2)^{-1}\left(\begin{matrix}
\frac{B-1}{2}\\
0\\
\end{matrix}\right).
\end{equation}

By \eqref{setequation}, the fixed points given by \eqref{Fixpoint1} and \eqref{Fixpoint2} all belong to $T$.
First of all, we give a key lemma  for the main result.

\begin{lemma}\label{Intersection}
Let $A, B \in\mathbb{Z}$ satisfying $|A|\leq B\text{ and }~2|A|\geq B+3$, and let $T$ be the $p2$-crystile defined by \eqref{expand-map} and \eqref{digit} and $c(T)$ be the $\pi$-rotation of $T$. Then  $\sharp \big(T\cap c(T)\big)\geq 2$.
\end{lemma}
\begin{proof}

By \eqref{Fixpoint1}, we notice that for $2\leq p,q\leq B-2$
$$\text{Fix}(f_p)=-\text{Fix}(f_q) \text{ if } p+q=B-1.$$
This means that Fix$(f_p)$ and Fix$(f_q)$  are both in $T \text{ and }c(T)$.
If $B>3$, these points are different and we are done.
If $B\leq 3$, we only need to consider the case $|A|=3, B=3$ since we assume that $2|A|\geq B+3$. Since $B=3$, by \eqref{Fixpoint1}, Fix$(f_2)=\left(\begin{matrix}
0\\
0\\
\end{matrix}\right)$ 
which is in $T\cap c(T)$. And for the case $A=3, B=3$, there exists an eventually periodic sequence of edges (see Figure \ref{loop_1}).
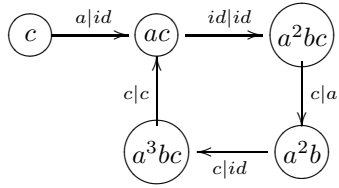
\begin{figure}[h] 
\hskip 0.9cm \xymatrix{
*++[o][F]{c}  
\ar[r]^(0.5){a|id}& 
*++[o][F]{ac}
 \ar@/^0ex/[r]^(0.5){id|id}& 
*++[o][F]{a^2 bc}
 \ar@/^0ex/[d]^(0.5){c|a}& \\
& 
*++[o][F]{a^3bc} 
\ar@/^0ex/[u]^{c|c}&
*++[o][F]{a^2b} 
\ar@/^0ex/[l]^{c|id}&
}
\caption{An eventually periodic sequence of edges for $A=3, B=3.$}\label{loop_1}
\end{figure}
\\
The edges of this figure are defined in the same way as in Definition \ref{graph} and it follows that 
$$x_0=\lim_{n\rightarrow\infty}g^{-1}a\circ (g^{-1}\circ g^{-1}c \circ g^{-1}c\circ g^{-1}c)^n(t)\in T\cap c(T,)$$
(see also Characterization \ref{B_r}). Here, $t\in \mathbb{R}^2$ is arbitrary. Note that
$$x_0=g^{-1}a\Big(\text{Fix}(g^{-1}\circ g^{-1}c \circ g^{-1}c\circ g^{-1}c)\Big),$$
and it is easy to compute that $x_0=\left(\begin{matrix}
-\frac{13}{73}\\
\quad\frac{16}{219}\\
\end{matrix}\right)
 \neq \left(\begin{matrix}
0\\
0\\
\end{matrix}\right)$. 

For the case $A=-3, B=3$, we find the eventually periodic sequence of edges 
\begin{figure}[h] 
\hskip 0.8cm \xymatrix{
*++[o][F]{c}  
\ar[r]^(0.5){a|c}& 
*++[o][F]{a}
   \ar@/^0ex/[r]^(0.45){a|id}& 
*++[o][F]{a^{-1}b}
 \ar@/^0ex/[r]^(0.45){a|c}& 
*++[o][F]{a^{-4}b^2c}
\ar`u[r] `[r]  `[]_(0.65){id|id} []& 
}
\end{figure}\\
So we have $$x'_0=\lim_{n\rightarrow\infty}g^{-1}a\circ g^{-1}a\circ g^{-1}a \circ (g^{-1})^n(t)\in T\cap c(T),$$
and  it is easy to verify that 
$x'_0=\left(\begin{matrix}
0\\
1\\
\end{matrix}\right)\neq \left(\begin{matrix}
0\\
0\\
\end{matrix}\right).$
\end{proof}

\begin{theorem}
Let $A, B \in\mathbb{Z}$ satisfying $|A|\leq B \text{ and } ~2|A|\geq B+3$, and let $T$ be the crystallographic replication tile defined by the data $(g,\mathcal{D})$ given in \eqref{expand-map} and \eqref{digit}. Then $T$ is not disk-like.
\end{theorem}
\begin{proof}
By a result of \cite{LauLeung07}, we know that if $2|A|\geq B+3$, then $T^{\ell}$ is not disk-like. Suppose that $T$ is disk-like.
By Lemma \ref{Intersection}, we have $\sharp(T\cap c(T))\geq 2$. By \cite[Proposition 4.1 item (2), p. 127]{LoridantLuoThuswaldner07}, this implies that $T\cap c(T)$ is a simple arc. Therefore  $T\cup c(T)$ is disk-like,  as the union of two topological disks whose intersection is a simple arc is again a topological disk. However, by Lemma \ref{twotiles}, $T^{\ell}$ is a translation of $T\cup c(T)$, therefore $T^{\ell}$ must be disk-like. This contradicts the assumption $2|A|\geq B+3$.
\end{proof}

\begin{figure}[h]
\centering
\subfigure[The lattice tile]{\includegraphics[width=2.3 in]{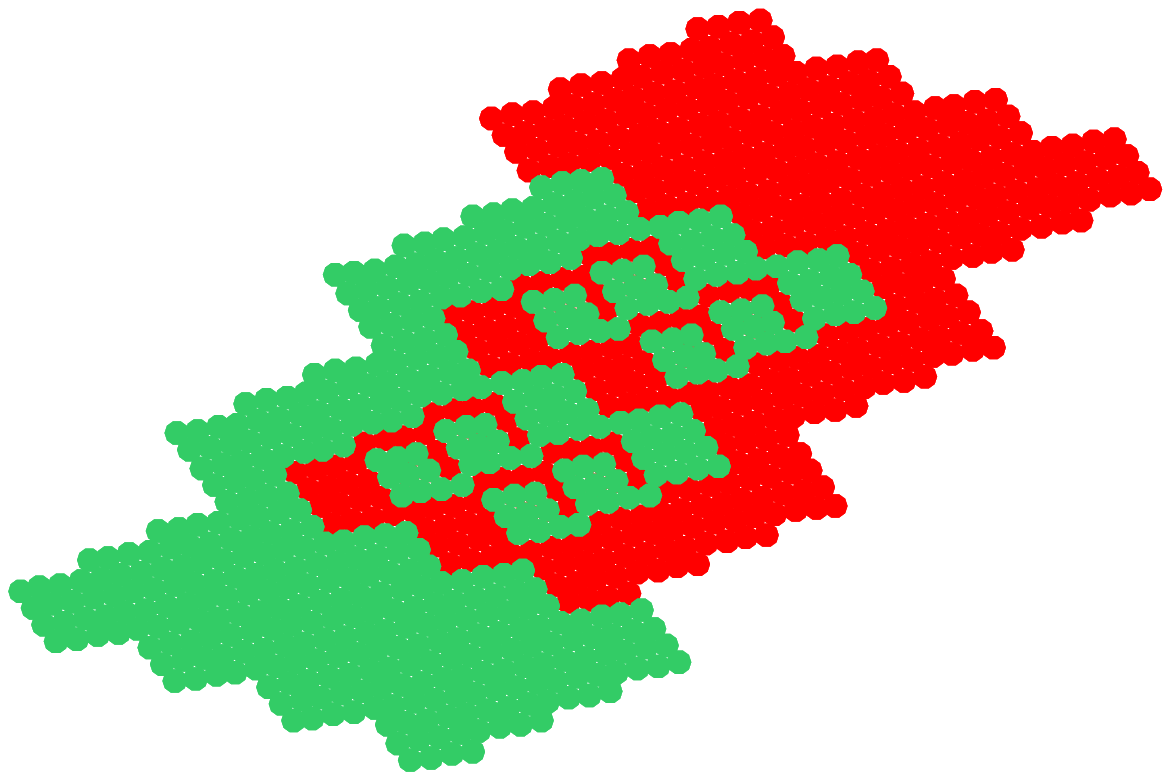}}\label{A1B3_firstsub}
\subfigure[The crystallographic tile]{\includegraphics[width=2.3 in]{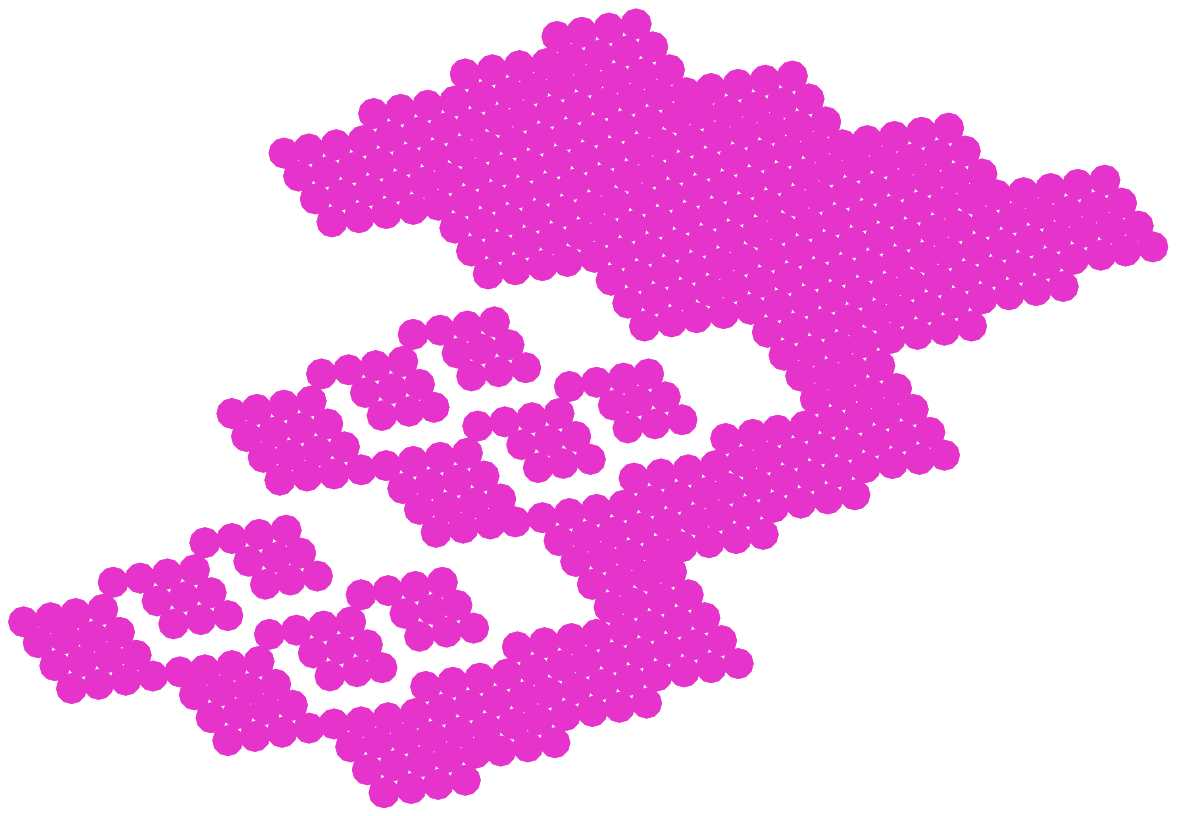}}\label{A1B3_secondsub}
\caption{$A=1, B=4$.}\label{A1B3_subfigures}
\end{figure}
\begin{figure}[h]
\centering
\subfigure[The lattice tile]{\includegraphics[width=2.3 in]{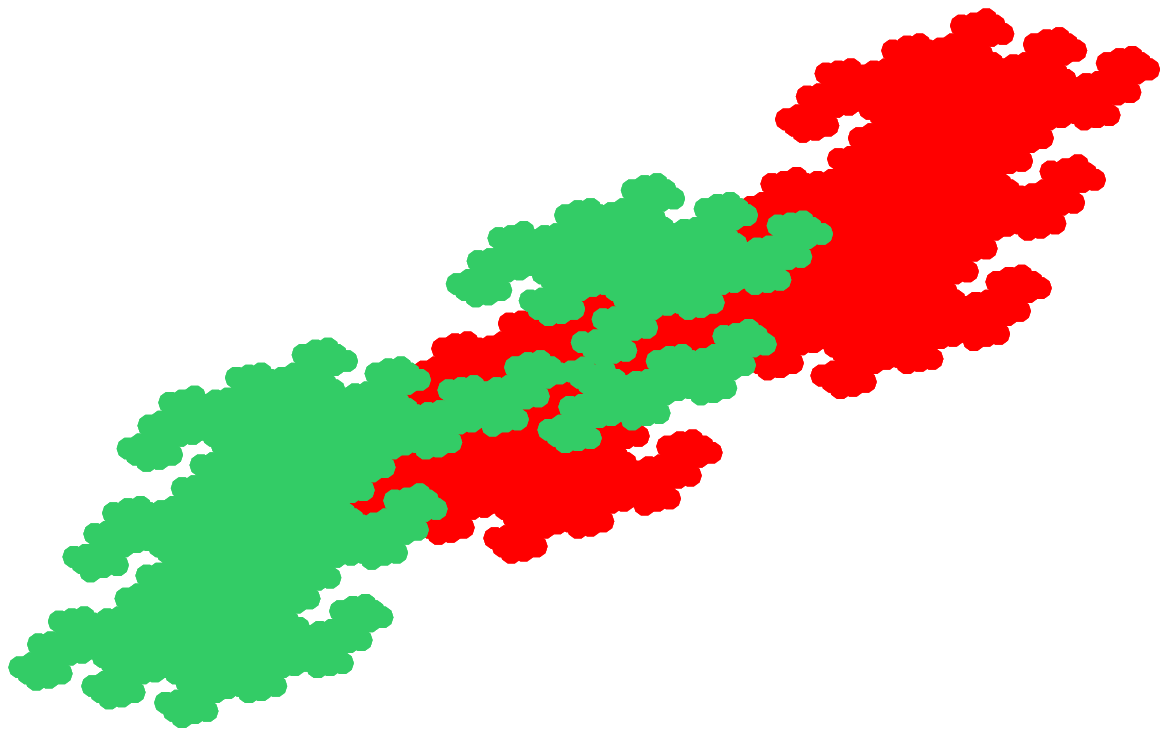}}\label{A2B3_firstsub}
\subfigure[The crystallographic tile]{\includegraphics[width=2.3 in]{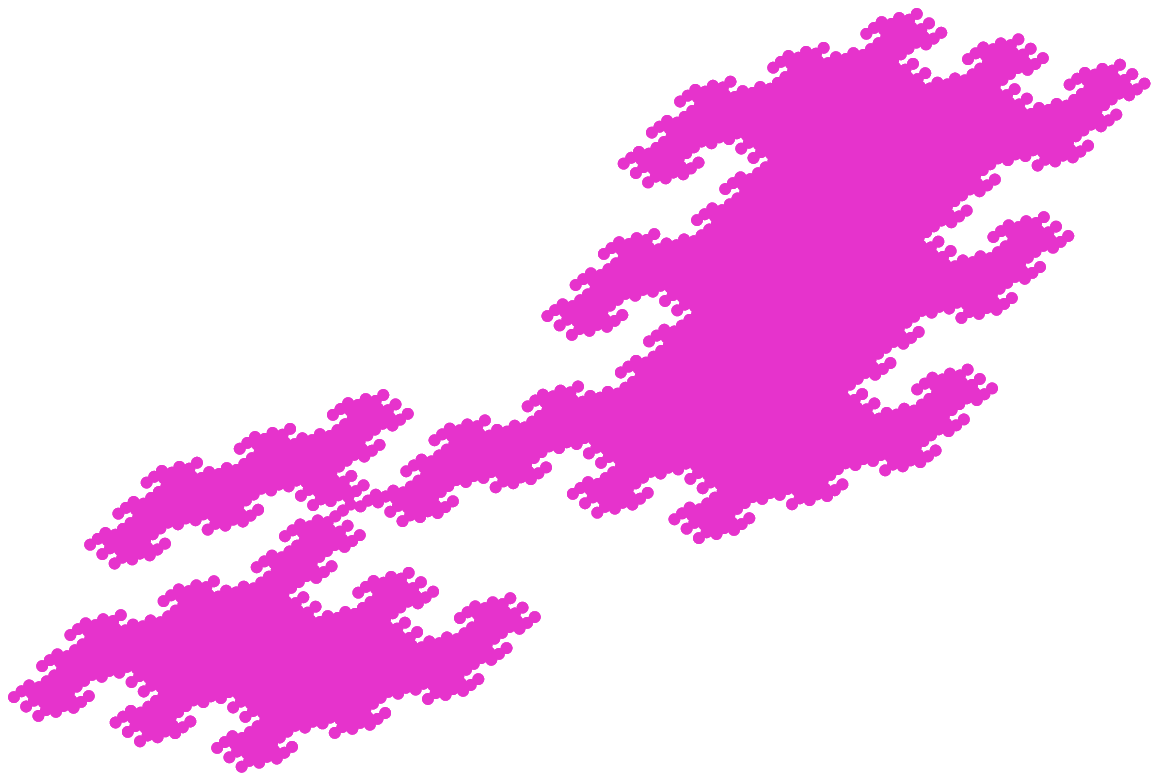}}\label{A2B3_secondsub}
\caption{Lattice tile and Crystile for $A=2, B=3$.}\label{A2B3_subfigures}
\end{figure}
\begin{figure}[h]
\centering
\subfigure[The lattice tile]{\includegraphics[width=2.3 in]{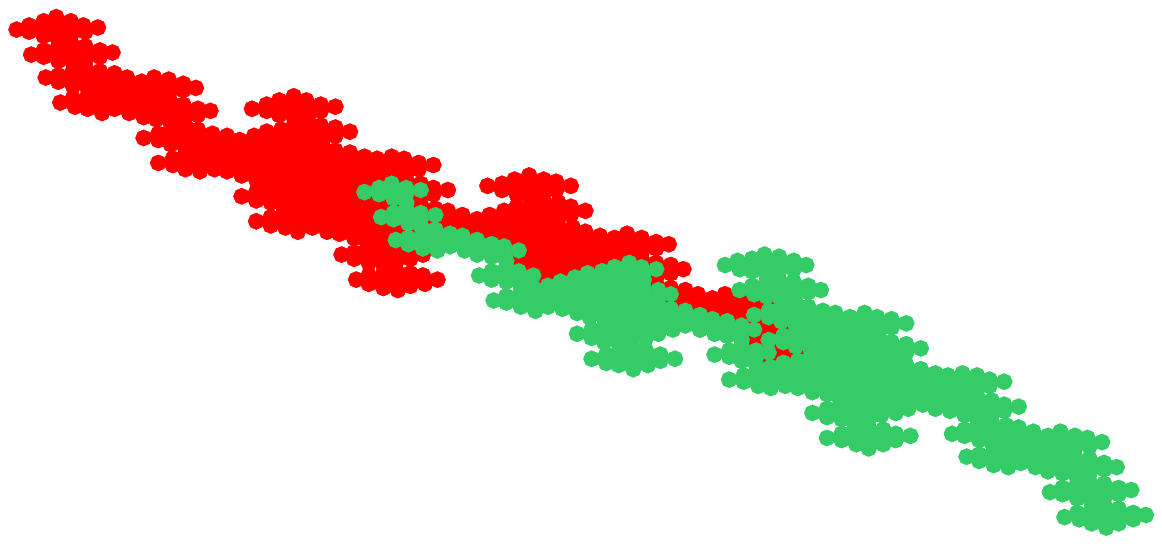}}\label{A3B4_firstsub}
\subfigure[The crystallographic tile]{\includegraphics[width=2.3 in]{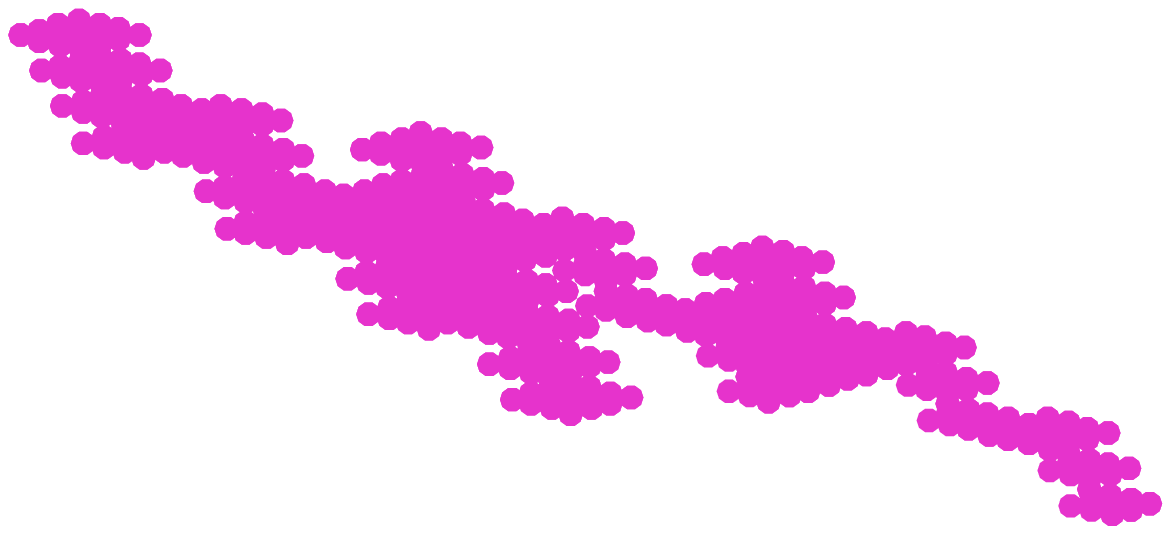}}\label{A3B4_secondsub}
\caption{Lattice tile and Crystile for $A=-3, B=4$.}\label{A3B4_subfigures}
\end{figure}
\begin{figure}[h]
\centering
\subfigure[The lattice tile]{\includegraphics[width=2.3 in]{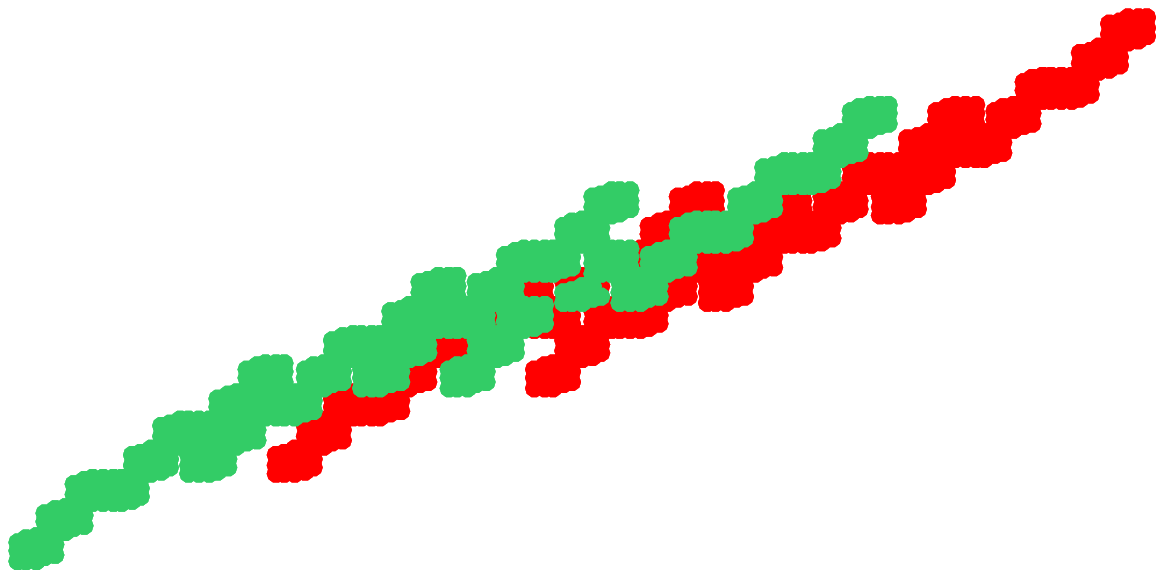}}\label{A3B3_firstsub}
\subfigure[The crystallographic tile]{\includegraphics[width=2.3 in]{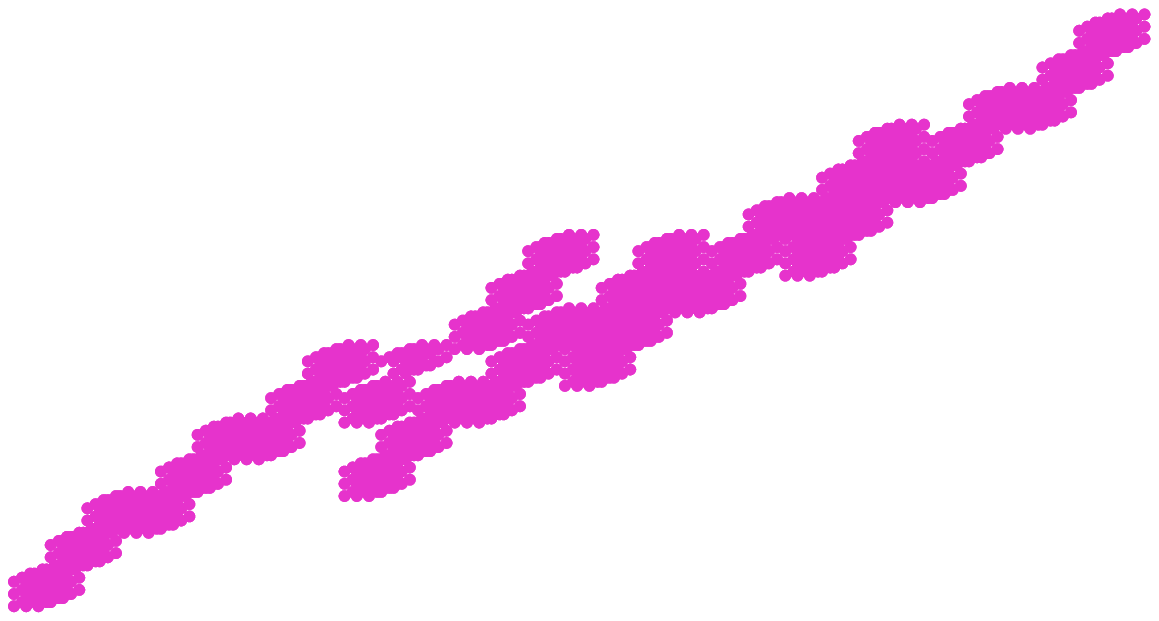}}\label{A3B3_secondsub}
\caption{Lattice tile and Crystile for $A=3, B=3$.}\label{A3B3_subfigures}
\end{figure}
\section{Examples}\label{sec_example}

Now we provide some examples.  For fixed $A \text{ and } B$, even though the lattice tile $T^{\ell}$ is a translate of $T\cup(-T)$, $T$ and $T^{\ell}$ may have completely different topological behaviour. We give the following examples to illustrate this phenomenon. In Figure \ref{A1B3_subfigures}, $A=1, B=4$, $T$ and $T^{\ell}$ are both disk-like. For Figure~ \ref{A2B3_subfigures} and Figure \ref{A3B4_subfigures}, $T^{\ell}$ is disk-like while $T$ is not. In Figure \ref{A3B3_subfigures}, $T$ and $T^{\ell}$ are both not disk-like.






\bigskip

\bigskip

\bigskip

\bigskip

\bibliographystyle{siam}   
\bibliography{biblio}

\end{document}